\documentclass[pdflatex,sn-mathphys,Numbered]{sn-jnl}

\usepackage{graphicx}
\usepackage{subfig}
\usepackage{amsmath,amssymb,amsfonts}
\usepackage{mathtools}
\usepackage{amsthm}
\usepackage{xcolor}
\usepackage{textcomp}
\usepackage{manyfoot}
\usepackage{booktabs}
\usepackage{placeins}

\theoremstyle{plain}
\newtheorem{theorem}{Theorem}[section]
\newtheorem{lemma}[theorem]{Lemma}
\newtheorem{corollary}[theorem]{Corollary}

\theoremstyle{definition}
\newtheorem{definition}[theorem]{Definition}

\theoremstyle{remark}
\newtheorem{remark}[theorem]{Remark}

\begin{document}

\title[Sparse space-time spectral methods can time-step]{Sparse space-time spectral methods can time-step by peel and pass}

\author*[1]{\fnm{Timon S.} \sur{Gutleb}}\email{T.S.Gutleb@leeds.ac.uk}

\affil[1]{\orgdiv{School of Computer Science}, \orgname{University of Leeds}, \orgaddress{\city{Leeds}, \country{UK}}}

\abstract{Global space-time spectral methods give spectral accuracy in time but typically require the whole space-time history to be resolved and stored on a single tensor-product domain $T \times \Omega$. We record that in an endpoint-benign Legendre or Chebyshev-$T$ time basis, whose polynomials all equal one at the right endpoint, the final time slice of a space-time block is recovered exactly by summing the stored coefficients along the time index. This peel-and-pass step is a special case of a Jacobi endpoint identity, which also gives derivative formulae for higher-order equations. Writing such higher-order equations as first-order systems preserves the benign value-passing structure. The result is a sparse space-time spectral element method that advances block by block, stores only one block, and needs far fewer time coefficients per solve for long-time problems. We prove the identities, give resident-memory, solve-cost and error-propagation models, and demonstrate the method on $(1{+}1)$D heat, wave and Klein--Gordon equations, and on $(2{+}1)$D fractional heat on the disk with weighted Zernike polynomials in space.}



\maketitle
\hypersetup{bookmarksdepth=3}

\section{Introduction}
Spectral methods approximate the solution of a differential equation by a truncated expansion in a global basis, converging rapidly for sufficiently well-behaved solutions. The ultraspherical spectral method of Olver and Townsend \cite{olver2013fast} represents the solution in a Gegenbauer or, more generally, Jacobi basis and expresses differentiation and multiplication as \emph{sparse, banded} operators between these bases, in contrast to the dense and increasingly ill-conditioned matrices of classical collocation. This pairing of spectral accuracy with banded, well-conditioned linear algebra \cite{olver2020fast} has since been carried to multivariate and non-trivial geometries by constructing sparse-operator bases of orthogonal polynomials adapted to triangles \cite{olver2019sparse}, disks and disk slices \cite{vasil_burns_lecoanet_olver_brown_oishi_2016_polar,snowball_olver_2019_disk_slices}, spherical caps \cite{snowball_olver_2020_spherical_caps}, algebraic curves \cite{fasondini_olver_xu_2022_curves,fasondini2023orthogonal,yao2026sparse}, annuli and cylinders \cite{papadopoulos2024building,papadopoulos_olver_2025_disk_annuli}, and generalised non-classical orthogonal polynomials \cite{gutleb_olver_slevinsky_2023_measure,gumerov_rigg_slevinsky_2026_measure_modification}. Decomposing a complicated spatial domain into sparse spectral elements of this kind is likewise established \cite{fortunato2021ultraspherical,olver2019sparse,papadopoulos2024building,papadopoulos_olver_2025_disk_annuli}.\\
When the differential equation is time-dependent, the dominant strategy is the method of lines: discretise space and march the resulting system of ordinary differential equations with a time stepper \cite{hesthaven2007spectral,cheng2023solving}. An appealing alternative is to place time on the same footing as space, expanding the solution in a global polynomial basis in time as well and solving the whole space-time problem at once. Earlier spectral-in-time constructions include time-spectral methods for hyperbolic and parabolic equations \cite{talezer1986hyperbolic,talezer1989parabolic,ierley_spencer_worthing_1992_time_spectral}, as well as single- and multi-interval Legendre methods in time \cite{tang_ma_2002_multi_interval}. The authors in \cite{scheffel2017optimizing} use the generalised weighted residual method with Chebyshev temporal subdomains to improve sparsity and memory use. More recent space-time spectral methods include Legendre and Chebyshev collocation schemes for time-dependent PDEs \cite{lui2017legendre,lui_nataj_2020_chebyshev,lui_nataj_2021_linear}, a Stokes formulation and thesis treatment by Kaur and Lui \cite{kaur_space-time_2022,kaur_lui_stokes_2023}, and linear and nonlinear examples by Kaur, Lui, Nataj and Wilegoda Liyanage \cite{kaur_lui_nataj_wilegoda_2025}. Related spectral-in-time constructions have also been used for fractional problems \cite{pu_fasondini_2023_fractional}.\\  The Achilles heel of global space-time spectral methods, and the reason they are seldom used for long-time or higher-dimensional solves, is that treating time as another coordinate raises the dimension of the discretised problem by one. The unknowns, operators and solver work then live on a full space-time tensor rather than on a spatial state advanced by a time stepper, increasing both computational complexity and resident memory. For persistent (e.g. oscillatory) or otherwise time-resolved dynamics this cost is compounded by time horizon: the number of coefficients needed on a single global interval grows with the interval length, as the method must store and solve for the whole history of the solution on one space-time block.\\
This approach has thus been described as not allowing time stepping \cite{kaur_space-time_2022}. In this paper we record that this obstruction belongs to a specific single-global-block formulation, not to space-time spectral methods themselves. For a Jacobi time basis, the solution on the final time slice of a space-time block is recovered from the stored coefficients by a simple endpoint contraction along the time index. In computation we focus on the endpoint-benign Legendre and Chebyshev-$T$ cases, where this contraction is nothing more than a summation of coefficients. This lets us ``peel off'' the final slice of one block in coefficient space and pass it as the initial condition of the next, decomposing the time axis into elements. The method then holds only one space-time block in memory at a time, reducing the resident memory from that of the whole history to that of a single block while preserving spectral accuracy and improving the geometric rate on each shorter block.

\section{Sparse Space-Time Blocks}
\subsection{Orthogonal Polynomial Bases}\label{subsec:opmethods}
Let $\{p_k\}_{k\ge 0}$ be a family of orthogonal polynomials on a given domain. Sparse spectral methods approximate functions $f\approx\sum_k f_k p_k$ by truncation and represent differentiation, basis conversion and multiplication by sparse operators between compatible polynomial bases \cite{olver2013fast,olver2020fast}. After boundary and initial rows are appended, this gives sparse bordered linear systems. We shall make no assumption on the spatial domain $\Omega$ beyond the availability of such a basis and the associated sparse operators.
\subsection{Space-Time Operators}\label{subsec:space-time}
Let $\Omega \subset \mathbb{R}^M$ be the spatial domain, resolved by a complete orthogonal-polynomial basis $\{q_\mathbf{k}(\mathbf{x})\}$ of the kind just described, and let $T = [0, T_{\mathrm{end}}]$ be the time interval. A space-time spectral method treats time as an additional coordinate and seeks the solution on the product domain $T \times \Omega \subset \mathbb{R}^{M+1}$ in the product basis
\begin{align*}
    u(t, \mathbf{x}) = \sum_{j, \mathbf{k}} u_{j,\mathbf{k}}\, p_j(t)\, q_\mathbf{k}(\mathbf{x}),
\end{align*}
where $\{p_j\}$ is a univariate orthogonal-polynomial basis on $T$. The space-time differential operator is assembled on this product basis and bordered with the spatial boundary conditions and temporal initial condition(s), cf. the ultraspherical spectral method \cite{olver2013fast,olver2020fast}. For example, the heat operator $\partial_t-\partial_{xx}$ on a rectangular space-time block has the coefficient-space form
\begin{align*}
    S_x \otimes D_t - D_x^{(2)} \otimes S_t,
\end{align*}
up to affine scaling factors, where the factors are sparse differentiation and conversion matrices, cf. \cite{olver2020fast}. Boundary and initial conditions add bordered rows obtained from endpoint evaluation rows tensored with identities. For the Legendre--Legendre discretisations used in our numerical experiments, Figure \ref{fig:operator-sparsity} shows representative sparsity patterns.\\
\begin{figure}\centering
     \subfloat[Heat equation]
    {{ \centering \includegraphics[width=0.31\textwidth]{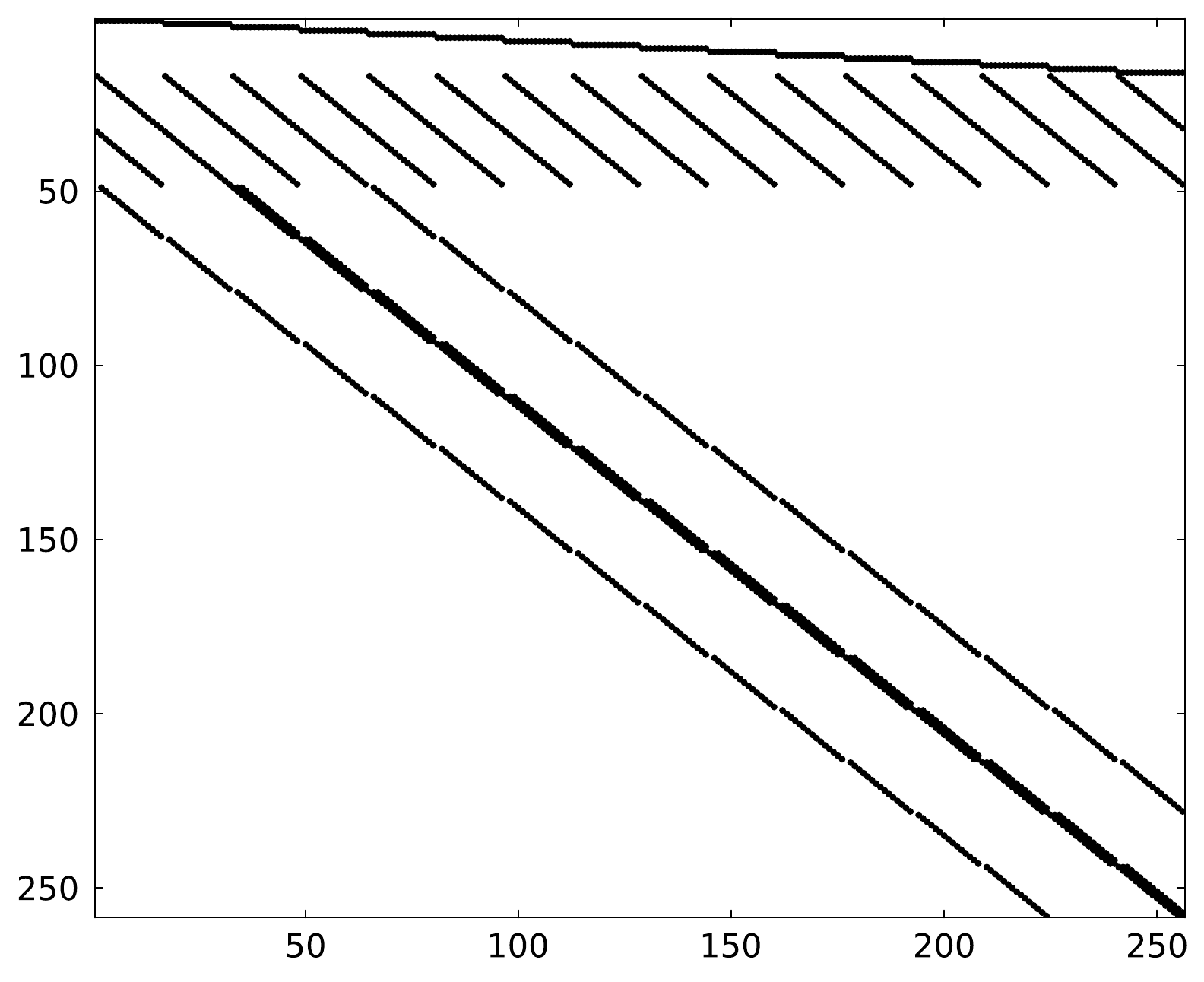} }}
     \subfloat[Wave system]
    {{ \centering \includegraphics[width=0.31\textwidth]{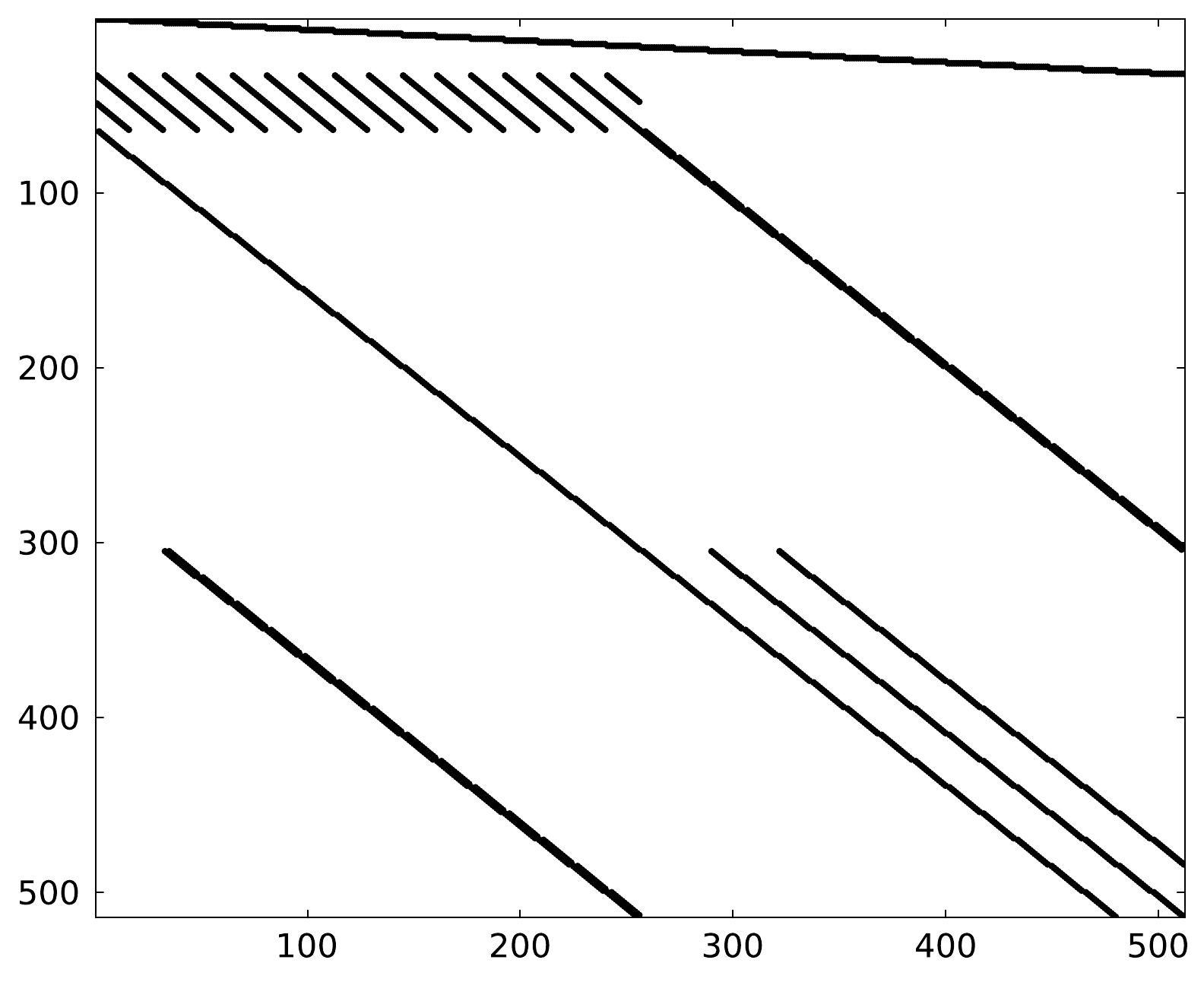} }}
     \subfloat[Klein--Gordon system]
    {{ \centering \includegraphics[width=0.31\textwidth]{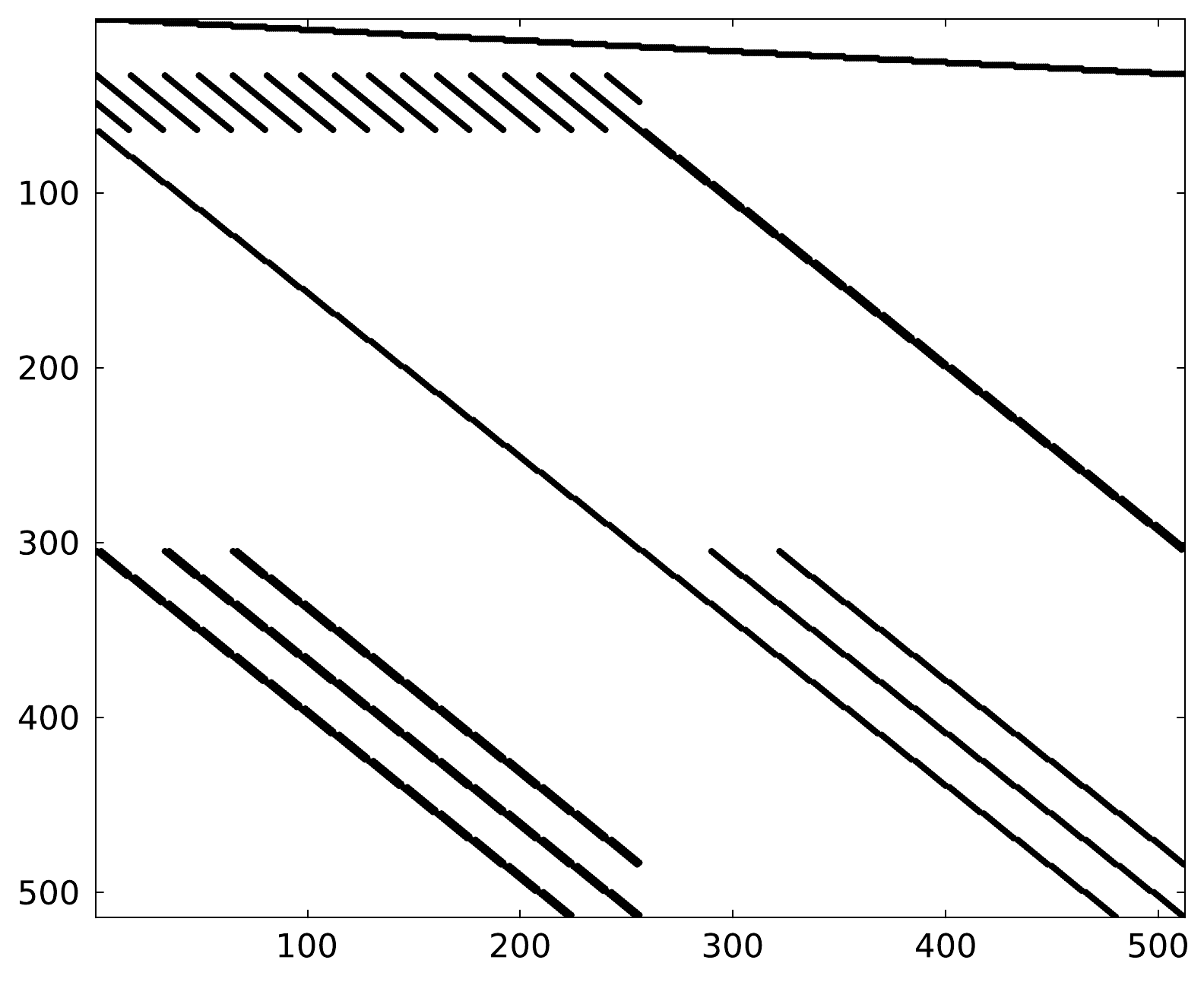} }}
    \caption{Nonzero patterns of single-block Legendre--Legendre operators ($N=16$ per coordinate) for the heat, first-order wave, and Klein--Gordon ($m^2=20$) systems of Section \ref{sec:experiments}. Top rows are initial and boundary conditions; the rest are PDE rows.}
    \label{fig:operator-sparsity}
    \end{figure}
This is spectrally accurate in time but pays for it in memory and computational complexity. For nontrivial long-time dynamics, a single global time expansion needs more coefficients as the time window grows, and the coefficient tensor $u_{j,\mathbf{k}}$ encodes the solution at every instant simultaneously. For deep-time integration, and acutely in $(3{+}1)$ dimensions where the spatial factor is already large, the global space-time approach can therefore become impractical.\\
Our contribution is not to reduce the intrinsic dimensionality of the space-time discretisation. Each block remains an $(M{+}1)$-dimensional problem. The gain is instead to replace one large global space-time solve by a sequence of smaller solves, as sketched in Figure~\ref{fig:schematic}, each needing substantially fewer time coefficients than the full interval.
\section{Peel and Pass}\label{sec:peelpass}
Domain decomposition is central to sparse spectral element methods. As discussed in the introduction, spectral intervals in time and slab-wise space-time formulations also have precedents. Among them, the generalised weighted residual method with Chebyshev temporal subdomains \cite{scheffel2017optimizing} is closest in spirit: the endpoint solution from one time interval supplies the data or initial vector for the next, though its sparse matrices arise from local coupling in the residual/Jacobian system, rather than from ultraspherical-style differentiation and conversion operators acting directly on coefficient vectors.\\
We therefore do not claim that block-by-block space-time marching, or the use of spectral intervals in time, is new in itself. The distinction is that the sparse coefficient-space block itself supplies the interface data by endpoint contraction of the solved coefficients. No coupled multi-interval system, lifting construction, weak interface condition, quadrature, or re-expansion is needed. In endpoint-benign Legendre and Chebyshev-$T$ bases, value passing is simply summation along the time index. The endpoint values of Jacobi polynomials are of course classical. What is isolated here is their use as the interface mechanism itself, so that the solved block furnishes the next block's initial data exactly. The simplicity of this identity is precisely what makes it agnostic to the spatial discretisation, so that it composes with arbitrary sparse orthogonal-polynomial bases on non-trivial geometries. This turns the sparse polynomial space-time solve into a memory-bounded sequential time-stepper (or more accurately slab-marcher) while keeping the spectral-in-time accuracy and sparse space-time operator structure intact. Each step may overwrite the previous space-time block and keep only the data needed to initialise the next one.\\
\begin{figure}\centering
     \subfloat[]
    {{ \centering \includegraphics[width=0.48\textwidth]{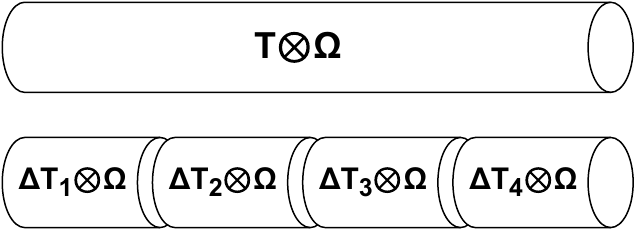} }}
     \subfloat[]
    {{ \centering \includegraphics[width=0.48\textwidth]{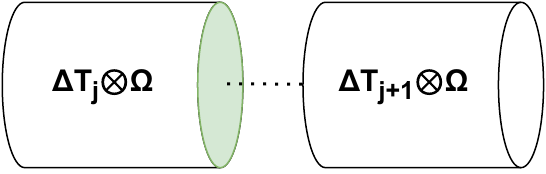} }}
    \caption{Decomposition of $T \times \Omega$ into time blocks $\Delta T_j \times \Omega$, and peel and pass of the final time layer, with derivative data obtained by endpoint contractions when needed.}
    \label{fig:schematic}
    \end{figure}
We first record the endpoint identity behind peel and pass for a general time basis, then specialise it to Jacobi polynomials. The general formula is useful because it identifies the endpoint weights explicitly, but it also exposes a practical issue: for some Jacobi parameters these weights grow with degree and may amplify high-degree coefficient noise. One could rescale any Jacobi basis to normalise a chosen endpoint functional, but this merely changes the coefficient coordinates with no advantage over the naturally endpoint-benign Legendre or Chebyshev-$T$ bases. We therefore single out Legendre and Chebyshev-$T$ as the preferred bases for the rest of the paper, since their value-passing weights are identically one.

\medskip
\begin{lemma}[Jacobi endpoint value identity]\label{lem:jacobi-endpoint-peelpass}
    Let a space-time block have temporal interval $[(n-1)\Delta t,n\Delta t]$, $n \in \mathbb{N}$, and spatial domain $\Omega \in \mathbb{R}^M$, and let $\{p_j^{(n)}\}$ be a polynomial basis on the temporal interval and $\{q_\mathbf{k}\}$ a polynomial basis on $\Omega$. Suppose that, on this block, we have
    \begin{align*}
        u(t,\mathbf{x}) = \sum_{j,\mathbf{k}} u_{j,\mathbf{k}} p_j^{(n)}(t) q_\mathbf{k}(\mathbf{x}).
    \end{align*}
    Then the final time slice of this block has the spatial-basis representation
    \begin{align*}
        u(n\Delta t,\mathbf{x}) = \sum_{\mathbf{k}} \left(\sum_{j} u_{j,\mathbf{k}} p_j^{(n)}(n\Delta t)\right) q_\mathbf{k}(\mathbf{x}).
    \end{align*}
    In particular, if $p_j^{(n)}(t)$ are shifted Jacobi polynomials $\tilde{P}_j^{(a,b),n}(t)$ with $a,b>-1$, then
    \begin{align*}
        u(n\Delta t,\mathbf{x}) = \sum_{\mathbf{k}} \left(\sum_{j} \binom{j+a}{j}\, u_{j,\mathbf{k}} \right) q_\mathbf{k}(\mathbf{x}) = \sum_{\mathbf{k}} \left(\sum_{j} \frac{(a+1)_j}{j!}\, u_{j,\mathbf{k}} \right) q_\mathbf{k}(\mathbf{x}),
    \end{align*}
    where $(a+1)_j$ denotes the Pochhammer symbol or rising factorial. In particular the peel-and-pass weights are independent of the block index $n$ and of the step size $\Delta t$.
\end{lemma}
\begin{proof}
    The first identity is the definition of the partial expansion in $\mathbf{x}$ obtained by collecting the time basis at $t = n\Delta t$. For the Jacobi case, let $m_n(t) = \tfrac{2t - (2n-1)\Delta t}{\Delta t}$ be the affine bijection mapping $[(n-1)\Delta t, n\Delta t]$ onto the reference interval $[-1,1]$, so that the shifted basis is $\tilde{P}_j^{(a,b),n}(t) = P_j^{(a,b)}(m_n(t))$. Since $m_n(n\Delta t) = 1$, evaluating at the final time gives $\tilde{P}_j^{(a,b),n}(n\Delta t) = P_j^{(a,b)}(1)$. By DLMF Table 18.6.1 \cite{NIST:DLMF} the right endpoint value of a Jacobi polynomial is $P_j^{(a,b)}(1) = \binom{j+a}{j} = (a+1)_j / j!$, which is independent of $b$, of $n$ and of $\Delta t$. Substituting into the general formula completes the proof.
\end{proof}

\medskip
\begin{corollary}[Legendre and Chebyshev endpoint values]\label{cor:legendre}
    Under the assumptions of Lemma \ref{lem:jacobi-endpoint-peelpass}, if the time basis consists of shifted Legendre polynomials ($a=b=0$) the peeled-off final slice is simply the time-summed expansion
    \begin{align*}
        u(n\Delta t,\mathbf{x}) = \sum_{\mathbf{k}} \left(\sum_{j} u_{j,\mathbf{k}} \right) q_\mathbf{k}(\mathbf{x}),
    \end{align*}
    The same plain summation holds for the Chebyshev basis of the first kind.
\end{corollary}
\begin{proof}
    For Legendre polynomials this is immediate from Lemma \ref{lem:jacobi-endpoint-peelpass} with $a=b=0$, giving $P_j(1)=\binom{j}{j}=1$. The Chebyshev value also follows from $T_j(1)=1$, cf. DLMF Table 18.6.1 \cite{NIST:DLMF}.
\end{proof}
This endpoint contraction is exact in the following sense: given the solved block coefficients, it recovers the final-slice values with no quadrature or interpolation error. The passed slice, if truncated, still carries the block's own discretisation error, which the next block inherits as its initial data. We quantify this accumulation in Section \ref{sec:analysis}. Similar results hold for the coefficients of time derivatives, which represent one option for passing initial data for equations of higher order in time such as the wave equation.

\medskip
\begin{lemma}[Jacobi endpoint derivative weights]\label{lem:jacobi-derivative-weights}
    Let $\tilde{P}_j^{(a,b),n}(t)=P_j^{(a,b)}(m_n(t))$ be the shifted Jacobi polynomial on the block $[(n-1)\Delta t,n\Delta t]$, with $a,b>-1$ and $m_n(t)=\tfrac{2t-(2n-1)\Delta t}{\Delta t}$. For any derivative order $r\ge 0$,
    \begin{align*}
        \partial_t^r \tilde{P}_j^{(a,b),n}(n\Delta t)
        =
        \Delta t^{-r}(j+a+b+1)_r \binom{j+a}{j-r},
    \end{align*}
    with the convention that the expression is zero when $j<r$.
    Consequently, if
    \begin{align*}
        u(t,\mathbf{x})=\sum_{j,\mathbf{k}}u_{j,\mathbf{k}}\tilde{P}_j^{(a,b),n}(t)q_\mathbf{k}(\mathbf{x}),
    \end{align*}
    then the $r$-th time derivative on the final slice is obtained by the coefficient contraction
    \begin{align*}
        \partial_t^r u(n\Delta t,\mathbf{x})
        =
        \sum_{\mathbf{k}}
        \left(
        \sum_j
        \Delta t^{-r}(j+a+b+1)_r \binom{j+a}{j-r} u_{j,\mathbf{k}}
        \right)
        q_\mathbf{k}(\mathbf{x}).
    \end{align*}
\end{lemma}
\begin{proof}
    The affine map has derivative $m_n'(t)=2/\Delta t$. Applying the Jacobi derivative identity DLMF 18.9.15 \cite{NIST:DLMF} repeatedly in the reference coordinate $\tau$ gives
    \begin{align*}
        \frac{d^r}{d\tau^r}P_j^{(a,b)}(\tau)
        =
        2^{-r}(j+a+b+1)_r P_{j-r}^{(a+r,b+r)}(\tau),
    \end{align*}
    with the derivative equal to zero for $j<r$. Evaluating at $\tau=1$ and using $P_{j-r}^{(a+r,b+r)}(1)=\binom{j+a}{j-r}$ gives the stated weight after multiplication by the chain-rule factor $(2/\Delta t)^r$.
\end{proof}

\medskip
\begin{corollary}[Legendre and Chebyshev endpoint derivative weights]\label{cor:endpoint-derivative-weights}
    Let $\tilde P_j^n(t)=P_j(m_n(t))$ and $\tilde T_j^n(t)=T_j(m_n(t))$ be the shifted Legendre and Chebyshev-$T$ bases on the block $[(n-1)\Delta t,n\Delta t]$. For Legendre,
    \begin{align*}
        \frac{d^r}{dt^r}\tilde P_j^n(n\Delta t)
        =
        \Delta t^{-r}\frac{(j+r)!}{r!(j-r)!},
    \end{align*}
    with the convention that the expression is zero when $j<r$.
    For Chebyshev-$T$, for $r\ge 1$, the weights are given in terms of the odd double factorial by
    \begin{align*}
        \frac{d^r}{dt^r}\tilde T_j^n(n\Delta t)
        =
        \left(\frac{2}{\Delta t}\right)^r
        \frac{j^2\prod_{m=1}^{r-1}(j^2-m^2)}{(2r-1)!!},
    \end{align*}
    again with the convention that it is zero when $j<r$. In particular, the first-derivative weights are $j(j+1)/\Delta t$ for Legendre and $2j^2/\Delta t$ for Chebyshev-$T$.
\end{corollary}
\begin{proof}
    The Legendre formula is Lemma \ref{lem:jacobi-derivative-weights} with $a=b=0$:
    \[
        \Delta t^{-r}(j+1)_r\binom{j}{j-r}
        =
        \Delta t^{-r}\frac{(j+r)!}{r!(j-r)!}.
    \]
    For Chebyshev-$T$, use the Jacobi representation
    \[
        T_j(x)=\frac{P_j^{(-1/2,-1/2)}(x)}{P_j^{(-1/2,-1/2)}(1)}.
    \]
    Applying Lemma \ref{lem:jacobi-derivative-weights} with $a=b=-1/2$ and dividing by $P_j^{(-1/2,-1/2)}(1)$ gives
    \[
        T_j^{(r)}(1)
        =
        2^{-r}(j)_r
        \frac{\binom{j-1/2}{j-r}}{\binom{j-1/2}{j}}
        =
        \frac{j^2\prod_{m=1}^{r-1}(j^2-m^2)}{(2r-1)!!}, \qquad r\ge 1.
    \]
    Here $(2r-1)!!=(2r-1)(2r-3)\cdots 3\cdot 1$ is the odd double factorial which appears after cancelling the half-integer binomial ratio into paired factors.
    Multiplying by the affine block factor $(2/\Delta t)^r$ gives the stated general weight.
\end{proof}
Crucially, both the value and derivative versions of peel and pass are exact endpoint contractions of the stored coefficient tensor along the single time index, so no re-expansion or quadrature is incurred at a block interface. Derivative passing nevertheless has weights that grow with degree even for Legendre and Chebyshev-$T$, and so can amplify high-degree coefficient error, especially in floating point arithmetic. For this reason, the numerical experiments below favour rewriting higher-order-in-time equations as first-order systems and passing values of the state components. 

The following cost estimate records the value-passing case used by the first-order formulations below.

\medskip
\begin{lemma}[Cost of peel-and-pass]\label{lem:peelcost}
Let a single space-time block have $N_t$ time coefficients and $N_x$ spatial coefficients, and index its stored real coefficients as $C_{j,k}$, where $0 \le j < N_t$ is the time index and $1 \le k \le N_x$ enumerates the retained spatial coefficient vector. Suppose that the time basis is either Legendre or Chebyshev-$T$, mapped affinely from the block to $[-1,1]$. 

Then passing the final value slice costs $\mathcal{O}(N_tN_x)$ floating-point operations and $\mathcal{O}(N_x)$ additional storage for the passed spatial coefficient vector. For fixed block resolution this is independent of final time and number of previous blocks.
\end{lemma}
\begin{proof}
Immediate from the formula
\[
    c_k^+ = \sum_{j=0}^{N_t-1} w_j C_{j,k}, \qquad 1 \le k \le N_x,
\]
where the endpoint weights $w_j$ are those of Corollary \ref{cor:legendre}. This performs one length-$N_t$ contraction for each of the $N_x$ spatial indices and stores only the resulting vector $c^+\in\mathbb{R}^{N_x}$.
\end{proof}
Passing a fixed number of additional endpoint traces, for example derivatives in a scalar higher-order formulation, only multiplies the work and storage by a fixed constant.
In floating point arithmetic, a defensive implementation of these endpoint contractions can sum from high to low degree, since for analytic solutions the spectral coefficients typically decay with degree.

\begin{remark}[Spatial composition and parallelism]\label{rem:spatialsem}
We treat the spatial domain $\Omega$ as a single spectral element in our experiments, but nothing in the construction requires this. Peel and pass acts only along the time index and is agnostic to how the spatial domain is discretised, so it composes directly with spatial spectral elements provided the spatial elements and their inter-element continuity conditions are assembled consistently within each space-time block. The endpoint contractions also \emph{parallelise} naturally in the spatial direction. At the coefficient level, the contractions for different spatial indices are independent. On $p$ workers and with balanced distribution over spatial indices, the value-passing parallel time is $\mathcal{O}(N_t\lceil N_x/p\rceil)$ up to scheduling and memory-bandwidth overheads. In a spatial element method, the same observation regarding parallelism applies element by element, with the passed interface data formed locally before the next block is assembled.
\end{remark}
\section{Memory, Cost, and Error Analysis}\label{sec:analysis-umbrella}
\subsection{Memory Requirements}\label{sec:memory}
We now make the memory advantage more precise. By \emph{resident memory} we mean the retained solve state: the spectral coefficients, the passed interface data, work arrays such as right-hand sides, and the retained factorisation used by the linear solver. Thus the memory count has a coefficient part and an operator factorisation part. The latter is usually dominant, and its scaling is governed by the fill-in of the factorisation.

\medskip
\begin{definition}[Coefficient ratio]\label{def:memory-ratios}
For a given global/block comparison, write $N_g$ and $N_b$ for the global and one-block unknown counts and set
\[
    R = \frac{N_g}{N_b},
\]
the factor by which the global solve carries more coefficients than a single block.\end{definition} 
For temporally resolved oscillatory solutions with fixed physical frequency content the global polynomial degree required on a single interval grows approximately linearly with the time horizon, so at matched temporal resolution density $R\approx L$ for $L$ uniform blocks. Other coefficient-growth laws give similar bookkeeping rules with different ratios.

\medskip
\begin{lemma}[Resident-memory scaling]\label{lem:memory}
Let the spatial discretisation use $N_x$ coefficients and compare a single global solve on $[0,T]$ with a uniform blocked solve. Let $M_g$ and $M_b$ denote the corresponding global and one-block resident-memory counts, measured in retained floating-point scalars as described above. With $R=N_g/N_b$ as in Definition \ref{def:memory-ratios}, the ratio $M_g/M_b$ has the following model scalings:
\begin{itemize}
    \item $M_g/M_b\sim R$ when only coefficient and linear work-array memory are counted;
    \item $M_g/M_b\sim R$ when a reusable retained factorisation has fixed bandwidth $p$;
    \item $M_g/M_b\sim R^2$ when a dense retained factorisation dominates the memory count.
\end{itemize}
\end{lemma}
\begin{proof}
The tensor-product unknown counts are $N_g=N_t(T)N_x$ and $N_b=N_t^{(b)}N_x$, where $N_t^{(b)}$ is the one-block time coefficient count, so $R=N_t(T)/N_t^{(b)}$. Coefficient arrays, right-hand sides and state work arrays are linear in the unknown count, so their ratio is $R$, up to fixed-size lower-order work arrays. One peeled and passed state overwrites the initial-data vector for the next block, so it does not require an additional resident vector in this count.

Let $S(N_g)$ and $S(N_b)$ denote the retained factor memory for the global and one-block solves. The counted resident memory is the linear coefficient/work-array part plus this factor memory. For a fixed-bandwidth factorisation, $S(N_g)\sim pN_g$ and $S(N_b)\sim pN_b$ with bandwidth $p$ independent of the problem size, so
\[
    \frac{M_g}{M_b}
    \sim
    \frac{N_g+pN_g}{N_b+pN_b}
    = R .
\]
For a dense factorisation, $S(N_g)\sim N_g^2$ and $S(N_b)\sim N_b^2$, giving
\[
    \frac{M_g}{M_b}
    \sim
    \frac{N_g+N_g^2}{N_b+N_b^2}
    \approx R^2
\]
when the quadratic factor term dominates.
\end{proof}
The fixed-bandwidth and dense-factor ratios above assume comparable fill-in behaviour between the global and block factorisations. If the large and small factors fill by substantially different fractions, finite-size measured ratios can move outside these model references.
\subsection{Solve Cost}\label{sec:cost}
The memory comparison is the main point of the construction, but the cost of the local solves is also important. When the left-hand-side block operator is the same on every uniform block, the block operator can be assembled and factorised once, then reused for each new right-hand side. Note that peel and pass itself does not care how the block operator was assembled: nonautonomous coefficients affect factorisation reuse, not the endpoint contraction, but they can affect trade-offs in solve cost.

\medskip
\begin{lemma}[Reusable-factorisation cost model]\label{lem:solve-cost}
    Suppose that one local space-time block has $N_b$ unknown coefficients, and assume that the left-hand-side block operator is reusable on every uniform block. Let $R=N_g/N_b$ be the coefficient ratio from Definition \ref{def:memory-ratios}. Let $F(N)$ denote the one-time cost of factorising an $N$-unknown operator, and let $A(N)$ denote the cost of applying that factorisation to one right-hand side. Compare $L$ uniform block solves, reusing a single block factorisation, against one global solve with $N_g=RN_b$ unknowns. The respective costs are
    \begin{align*}
        C_{\mathrm{block}}(N_b,L) = F(N_b) + L A(N_b),
        \qquad
        C_{\mathrm{global}}(N_b,R) = F(RN_b) + A(RN_b).
    \end{align*}
    If the factorisation and application costs are $F(N)=c_F N^\alpha$ and $A(N)=c_A N^\beta$, with $c_F,c_A>0$ and $\alpha\ge\beta$, the cost ratio is
    \begin{align*}
        \frac{C_{\mathrm{global}}(N_b,R)}{C_{\mathrm{block}}(N_b,L)}
        =
        \frac{c_F R^\alpha N_b^\alpha + c_A R^\beta N_b^\beta}
             {c_F N_b^\alpha + c_A L N_b^\beta}.
    \end{align*}
    When both global and block systems are refined with fixed coefficient ratio $R$ and fixed block count $L$, and $\alpha>\beta$, this ratio tends to $R^\alpha$.
\end{lemma}
\begin{proof}
    The cost identities follow immediately from the definitions of $F$ and $A$. Substituting $F(N)=c_F N^\alpha$ and $A(N)=c_A N^\beta$ gives the displayed ratio. Under proportional refinement with fixed $R$ and $L$, the $N_b^\alpha$ terms dominate numerator and denominator when $\alpha>\beta$, so the ratio tends to $c_F R^\alpha N_b^\alpha/(c_F N_b^\alpha)=R^\alpha$.
\end{proof}
The displayed ratio is $C_{\mathrm{global}}/C_{\mathrm{block}}$, so large values favour blocking over the global approach. Blocking replaces one factorisation at the global size by one reusable factorisation at the block size plus $L$ factor applications. Which timing ratio this produces depends on the solver regime and on the relation between the coefficient ratio $R$ and the number of blocks $L$. In the limiting fixed-bandwidth model, with bandwidth $p$ independent of $N_b$, taking $F(N)=c_Fp^2N$ and $A(N)=c_ApN$ gives
\begin{align*}
    \frac{C_{\mathrm{global}}}{C_{\mathrm{block}}}
    =
    \frac{R(c_Fp+c_A)}{c_Fp+c_A L}.
\end{align*}
This is the conservative lower-fill regime: if $R\approx L$ it lies between $1$ and $R$, but if $R$ is much smaller than $L$ repeated applications can dominate and reduce the speedup. For a dense factorisation, by contrast, the model $F(N)=c_F N^3$ and $A(N)=c_A N^2$ tends to $R^3$ under proportional refinement with fixed $R$ and $L$. The sparse-direct LU factorisations we explore in the numerical experiments section sit between these extremes: their effective powers and constants are set by realised fill-in, so the timing ratios are expected to track the sparse factor sizes and hence be comparable to the resident-memory ratios, up to ordering, cache, and implementation effects.

\subsection{Error Propagation Estimate}\label{sec:analysis}
Restricting an analytic solution from $[0,T]$ to a shorter block of length $\Delta t$ moves singularities farther away in the reference coordinate, enlarging the admissible Bernstein ellipse and improving the geometric rate of polynomial approximation. The bound below records this effect while accounting for per-block stability.

\medskip
\begin{definition}[Per-block stability factor]\label{def:block-stability}
    In a fixed norm, say that a blocked solve has per-block stability factor $K\ge 0$ if, on every block, the part of the outgoing final-slice error caused by the incoming initial-slice error is bounded by $K$ times that incoming error. Equivalently, for linear homogeneous block error-propagation maps $E_n$ from the incoming state error to the outgoing passed-slice error, one may take any $K$ satisfying
    \begin{align*}
        \|E_n \eta\| \le K\|\eta\| \qquad \text{for every block } n
    \end{align*}
    and every admissible incoming state error $\eta$.
    For a contractive problem such as heat flow, the continuous evolution has $K\le 1$ in the contractive norm. A discrete block solve likewise has $K\le 1$ when its block error-propagation map is non-expansive in that norm. For a quasi-contractive block evolution, $K\le e^{\omega\Delta t}=1+\mathcal{O}(\Delta t)$ for some $\omega\in\mathbb{R}$.
\end{definition}

\medskip
\begin{theorem}[Conditional error propagation for peel and pass]\label{thm:error}
    Consider a linear, well-posed evolution problem $\partial_t u = \mathcal{L}u$ on $\Omega$, measured in a fixed norm. For a second-order-in-time problem, rewrite the equation as a first-order system with state $U=(u,u_t)$. Advance the solution over $[0,T]$ in $L=T/\Delta t$ uniform blocks of size $\Delta t$, where $L \in \mathbb{Z}$, and use $N_t$ Legendre coefficients in time on each block. Let $e_0$ be the error in the initial state.

    Further, assume the following:
    \begin{enumerate}
        \item The exact solution is analytic in a complex neighbourhood of $[0,T]$ whose boundary is at least a distance $d>0$ from the real time interval.
        \item Peel and pass is exact, as in Corollary \ref{cor:legendre}, so that the block interface contributes no additional error.
        \item With exact initial data, the per-block final-slice discretisation error is at most
        \begin{align*}
            C\varrho_{\Delta t}^{-N_t}+\varepsilon_x,
        \end{align*}
        where $\varepsilon_x$ bounds the spatial truncation error on one block.
        \item In the chosen norm, the per-block space-time solve has a uniform stability factor $K \ge 0$ in the sense of Definition \ref{def:block-stability}.
    \end{enumerate}

    Then the error of the final passed slice obeys
    \begin{align*}
        \|u_h(T,\cdot) - u(T,\cdot)\| \quad\le\quad K^L\, \|e_0\| \quad+\quad \frac{K^L-1}{K-1}\Bigl(C\,\varrho_{\Delta t}^{-N_t} + \varepsilon_x\Bigr),
    \end{align*}
    where the fraction is read as $L$ when $K=1$. The constant $C$ depends on bounds for the analytic continuation of the exact solution in the chosen complex neighbourhood, but is independent of $N_t$, $\Delta t$, and $L$. The parameter $\varrho_{\Delta t}$ may be chosen to satisfy
    \begin{align*}
        1<\varrho_{\Delta t}<\delta + \sqrt{1+\delta^2}, \qquad \delta = \frac{2d}{\Delta t},
    \end{align*}
    and is then an admissible Bernstein-ellipse parameter for a function analytic at reference-coordinate distance at least $\delta$ from $[-1,1]$. More generally, one may use any admissible Bernstein parameter for the solution restricted to a single block.
\end{theorem}
\begin{proof}
    Write $e_n = \|u_h^{(n)}(t_n,\cdot) - u(t_n,\cdot)\|$ for the error of the slice passed out of block $n$, so $e_L$ is the quantity to be bounded. The block solver on block $n$ is supplied with the previous numerical slice, which by assumption (2) of peel and pass carries the error $e_{n-1}$ and no additional interface error. By the stability assumption (4), the portion of the outgoing error caused by this incoming error is at most $K e_{n-1}$. Solving the block with \emph{exact} initial data incurs the assumed per-block discretisation error $C\varrho_{\Delta t}^{-N_t}+\varepsilon_x$ of assumption (3). That assumption is the natural one for an analytic solution. Its time term has this form because the block-restricted exact solution is analytic on the Bernstein ellipse $E_{\varrho_{\Delta t}}$, so the truncation error of its $N_t$-term Legendre expansion decays like $\varrho_{\Delta t}^{-N_t}$ by the standard estimate \cite{trefethen2013approximation}. Its spatial term is the assumed bound $\varepsilon_x$. Combining the propagated incoming error and the local discretisation error,
    \begin{align*}
        e_n \le K e_{n-1} + C\varrho_{\Delta t}^{-N_t} + \varepsilon_x.
    \end{align*}
    Iterating this recursion from $e_0$ and summing the series $\sum_{k=0}^{L-1}K^k = (K^L-1)/(K-1)$ yields the stated bound. The displayed upper bound for $\varrho_{\Delta t}$ is the Bernstein parameter of the ellipse reaching reference-coordinate distance $\delta = 2d/\Delta t$ from $[-1,1]$. The strict inequality ensures that $E_{\varrho_{\Delta t}}$ remains inside the analytic neighbourhood after mapping the block $[t_{n-1},t_n]$ affinely onto $[-1,1]$. If the actual singularity geometry permits a larger ellipse, the same proof applies with that larger admissible parameter.
\end{proof}
For a dissipative problem with exact initial data and $K\le1$, this implies
\begin{align*}
    \|u_h(T,\cdot) - u(T,\cdot)\| \le L\Bigl(C\varrho_{\Delta t}^{-N_t}+\varepsilon_x\Bigr).
\end{align*}
The same propagation argument applies to inhomogeneous linear problems once the per-block consistency term is understood to include the discretisation of the forcing.
The accumulated temporal contribution is $L C\varrho_{\Delta t}^{-N_t}$. For fixed $T$ and $\Delta t=T/L$, admissible parameters may be chosen with $\varrho_{\Delta t}\sim c/\Delta t$ for any fixed $0<c<4d$ (since $\delta+\sqrt{1+\delta^2}\sim 2\delta = 4d/\Delta t$), giving
\[
    L C\varrho_{\Delta t}^{-N_t}
    =
    \mathcal{O}\!\left(\frac{T}{c^{N_t}}\left(\frac{T}{L}\right)^{N_t-1}\right),
    \qquad L\to\infty .
\]
Thus, for fixed final time, the accumulated temporal term decays like $(1/L)^{N_t-1}$. The accumulated spatial contribution in this bound is $L\varepsilon_x$, so convergence under block refinement is guaranteed only if the spatial error is controlled separately, for example by ensuring $\varepsilon_x=o(1/L)$.
\section{Numerical Experiments}\label{sec:experiments}
The numerical experiments were run in Julia \cite{bezanson2017julia} on an Apple M3 Pro MacBook Pro with 18 GB memory, using \texttt{ClassicalOrthogonalPolynomials.jl} \cite{classicalorthogonalpolynomials_github} and \texttt{MultivariateOrthogonalPolynomials.jl} \cite{multivariateorthogonalpolynomials_github}. The Julia code used to generate the numerical experiments and figures is provided in the companion repository archived at \cite{spacetimespectralpass_software}.

We compare a single global space-time solve with blocked peel-and-pass solves for heat, wave, Klein--Gordon, and disk fractional heat problems.
Reported solve times are medians after warmup. Following the ultraspherical spectral method \cite{olver2013fast,olver2020fast}, the highest-degree differential-operator rows are dropped and replaced by the boundary and initial conditions, giving a square block that is factorised by sparse LU. Each timed blocked run constructs the reusable block matrix, computes one sparse LU factorisation for each distinct block operator, and applies that factorisation through all blocks. Resident memory is counted in floating-point scalars: coefficient arrays, right-hand-side and state work arrays, and the nonzero values of the retained $L$ and $U$ factors.

Unless stated otherwise, ``error'' denotes pointwise max. abs. error at final time:
\begin{align*}
    \mathcal{E} = \max_{\mathbf{x} \in \Omega} \bigl| u_h(T,\mathbf{x}) - u(T,\mathbf{x}) \bigr|,
\end{align*}
where $u_h$ is the numerical and $u$ the true solution.

Table \ref{tab:numerical-summary} summarises the numerical experiments: the coefficient counts, timing and resident-memory ratios, and final-time errors at the first resolution where each curve has approximately reached its converged error. The following subsections give the problem descriptions and full convergence plots. The blocked columns in Table \ref{tab:numerical-summary} use the $L=4$ runs, except for the deep-time Klein--Gordon row, which uses $L=100$. The table reports the global and one-block time coefficient counts $N_t^{(g)}$ and $N_t^{(b)}$. Because the spatial truncation is fixed within each comparison, their ratio is the coefficient ratio $R$ from Definition \ref{def:memory-ratios}. The timing ratios are measured whole-run ratios and should be read in light of the reusable-factorisation model of Lemma \ref{lem:solve-cost}. The deep-time Klein--Gordon global solve cannot be run to resolution on the stated test machine, so its global entries are empty.

\begin{table}[t]
\centering
\caption{Summary of timing ratios, resident memory ratios and final-time errors.}
\label{tab:numerical-summary}
\begin{tabular}{lccccccc}
\toprule
Problem & $N_t^{(g)}$ & $N_t^{(b)}$ & $R$ & $T_g/T_b$ & $M_g/M_b$ & $\mathrm{err}_g$ & $\mathrm{err}_b$\\
\midrule
Heat & 24 & 16 & $1.50$ & $1.52$ & $3.00$ & $5.7{\times}10^{-14}$ & $6.4{\times}10^{-15}$\\
Heat, singular & 112 & 16 & $7.00$ & $16.99$ & $42.65$ & $2.4{\times}10^{-14}$ & $2.8{\times}10^{-14}$\\
Wave sys. & 48 & 24 & $2.00$ & $3.89$ & $4.72$ & $2.7{\times}10^{-14}$ & $6.7{\times}10^{-14}$\\
Klein--Gordon sys. & 80 & 32 & $2.50$ & $5.95$ & $6.40$ & $3.8{\times}10^{-13}$ & $2.6{\times}10^{-13}$\\
Klein--Gordon, deep & \textemdash & 96 & \textemdash & \textemdash & \textemdash & \textemdash & $2.1{\times}10^{-12}$\\
Disk heat, $\beta=1$ & 48 & 24 & $2.00$ & $2.58$ & $2.85$ & $2.9{\times}10^{-15}$ & $6.0{\times}10^{-14}$\\
Disk heat, $\beta=0.5$ & 48 & 16 & $3.00$ & $5.34$ & $5.77$ & $2.3{\times}10^{-16}$ & $2.0{\times}10^{-14}$\\
\bottomrule
\end{tabular}
\end{table}

\subsection{Heat Equation}
We compare a global solve with the time spectral element version on
\begin{align}\label{eq:heat1p1}
    u_t - u_{xx} =
    \left(\pi^2\left(1+\tfrac14\sin t\right)+\tfrac14\cos t\right)\sin(\pi x),
\end{align}
on $\Omega = (0,1)$ with zero Dirichlet conditions and initial data
\begin{align*}
    u(0,x) = \sin(\pi x),
\end{align*}
whose exact solution is
\begin{align*}
    u(t,x) = \left(1+\tfrac14\sin t\right)\sin(\pi x).
\end{align*}
We discretise each space-time block with tensor products of Legendre polynomials, using fixed spatial coefficient count $N_x=24$ and varying the time coefficient per block count $N_t$. The first-order-in-time, second-order-in-space operator is assembled together with the initial and Dirichlet boundary conditions as a bordered almost-banded system. We advance from $t=0$ to $T=10$ either as a single global block ($L=1$) or as $L=2$ and $L=4$ uniform blocks linked by peel and pass as in Corollary \ref{cor:legendre}.
Figure \ref{fig:heat1p1-fig1} plots the final-time error against time coefficient count, solve time, and resident memory.
\begin{figure}\centering
    \includegraphics[width=0.31\textwidth]{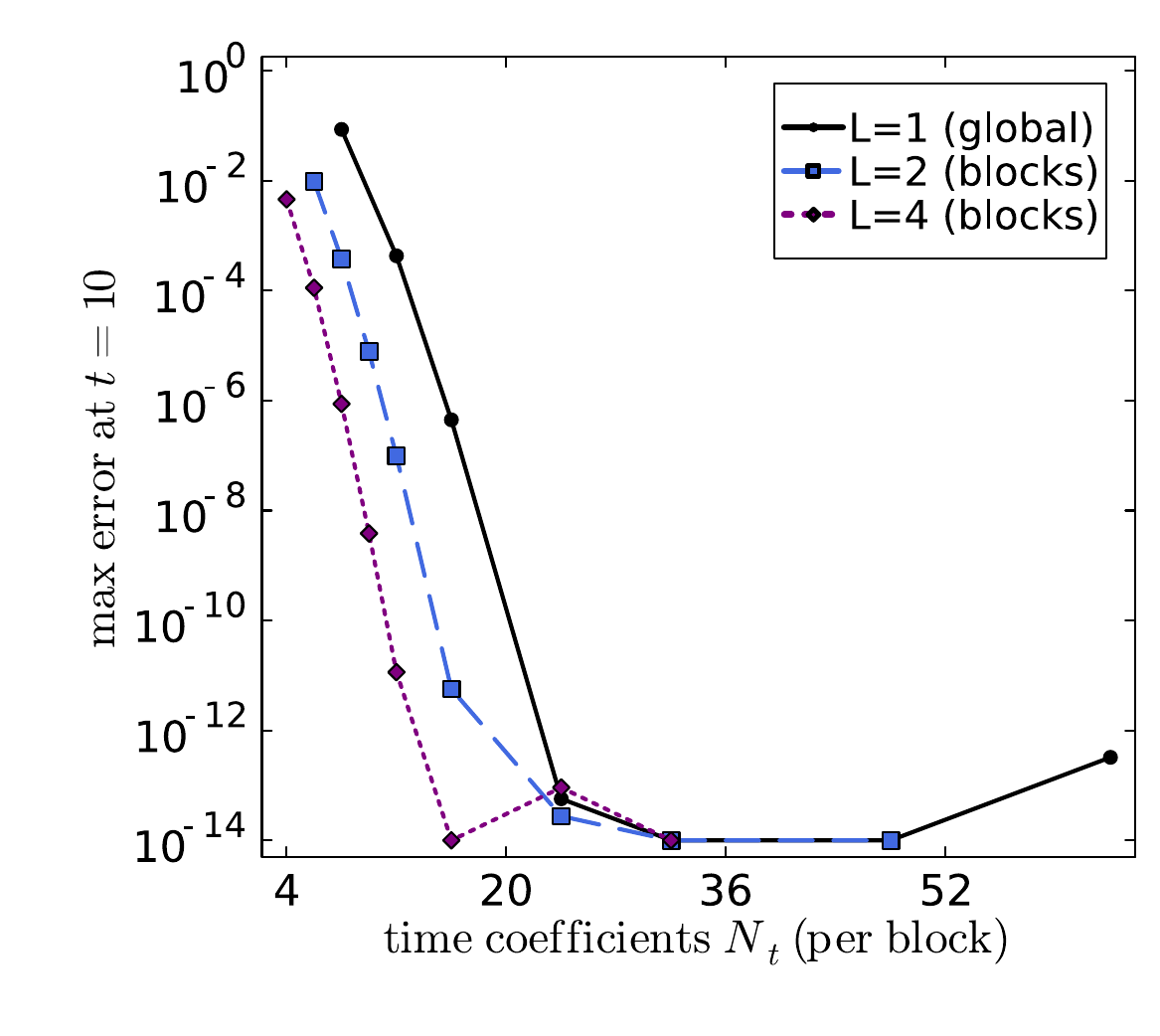}
    \includegraphics[width=0.31\textwidth]{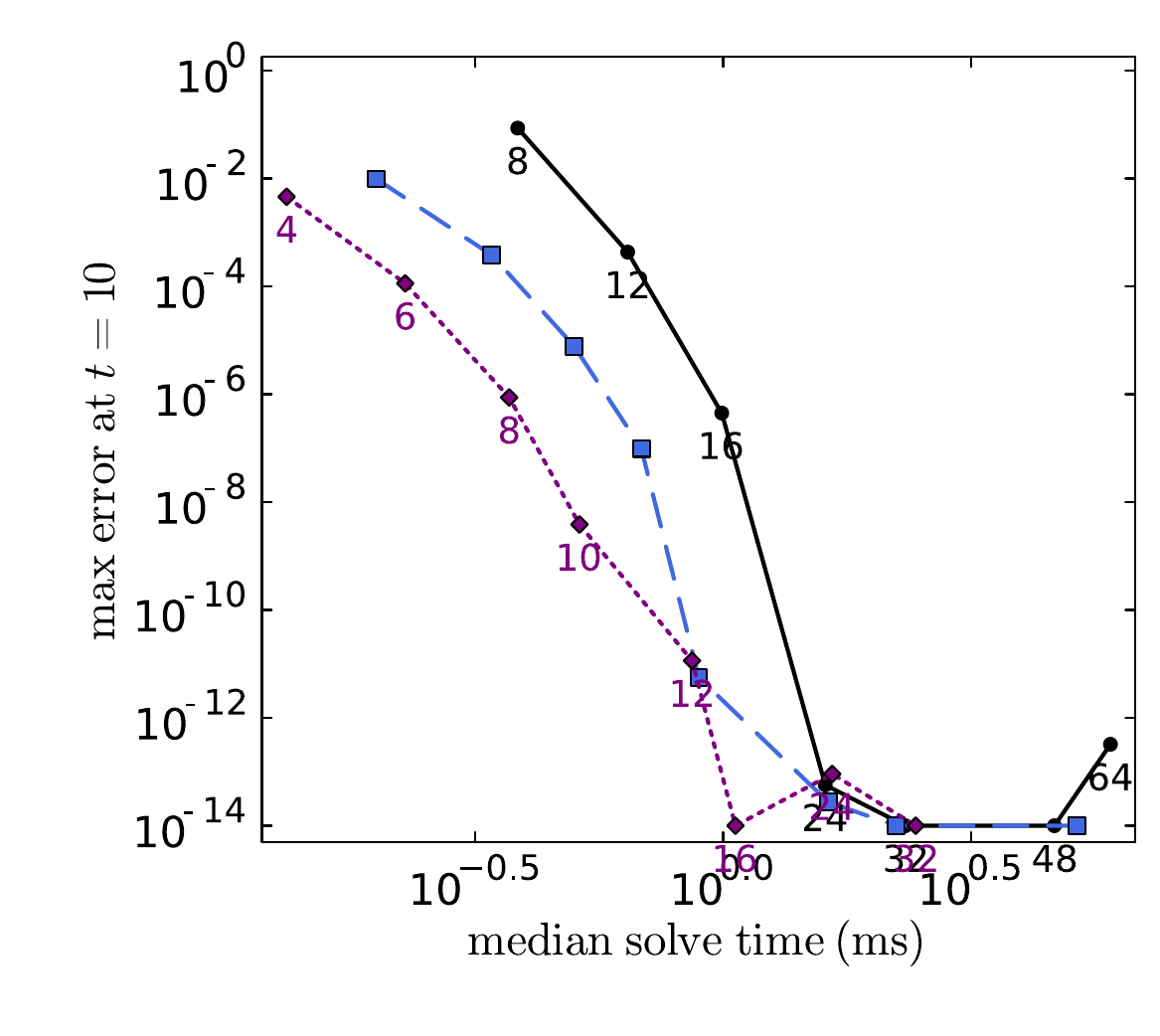}
    \includegraphics[width=0.31\textwidth]{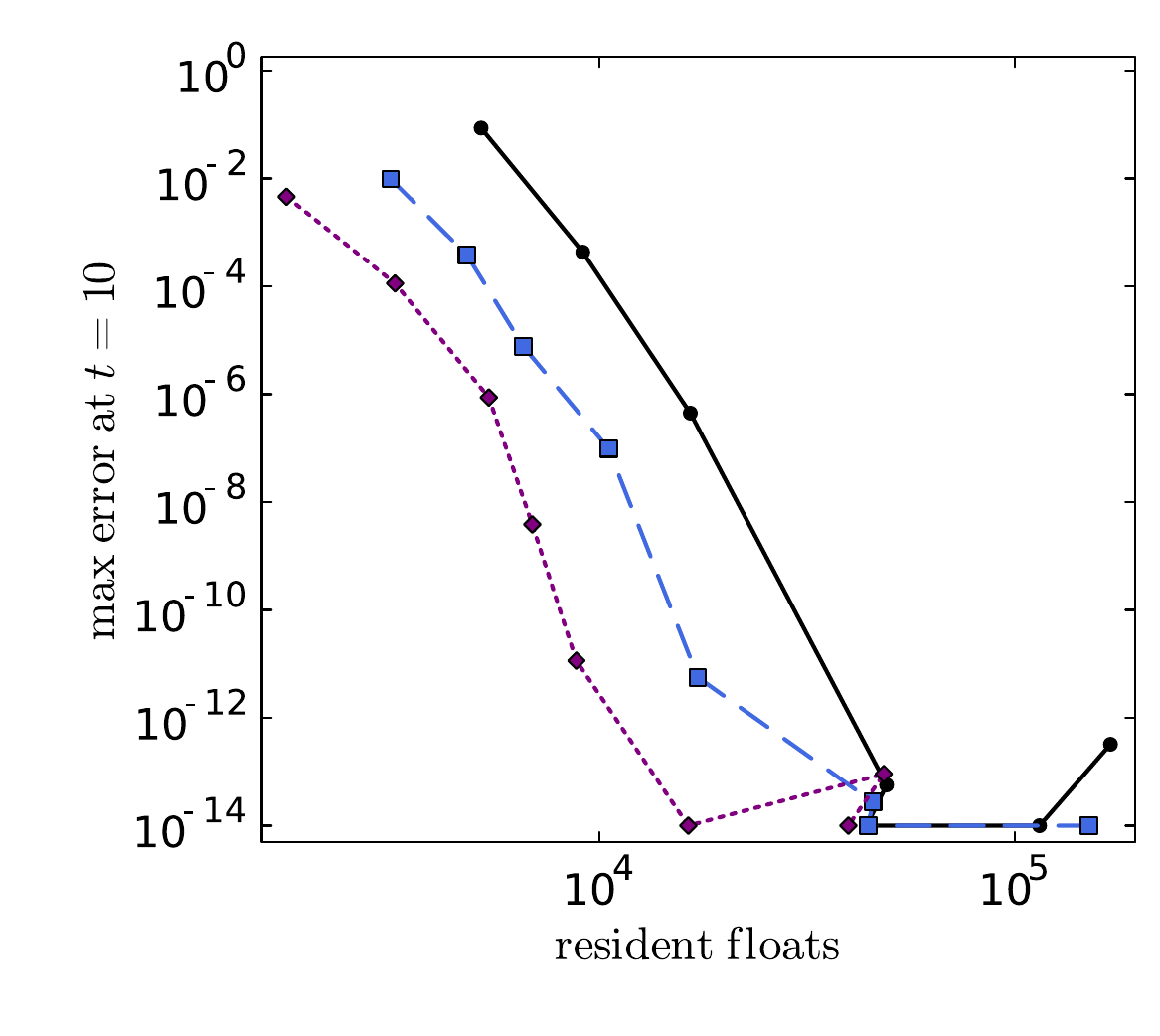}
    \caption{Global vs.\ time-spectral-element solves of the heat equation \eqref{eq:heat1p1} over $[0,10]$, $N_x=24$. Columns: final-time error vs.\ $N_t$ (per block), solve time ($N_t$ labelled at each marker), and resident memory; bottom-left is better.}
    \label{fig:heat1p1-fig1}
    \end{figure}
\subsection{Heat Equation with a Complex-Time Singularity}\label{subsec:heatpole}
The other problems in this section have solutions that are entire in time, so blocking mainly improves memory and solve cost. To isolate the Bernstein-ellipse rate mechanism of Theorem \ref{thm:error}, we repeat the heat experiment with a temporal profile that is analytic but \emph{not} entire. We retain the heat operator and spatial profile $\sin(\pi x)$, but take
\begin{align}\label{eq:heatpole}
    u(t,x) = g(t)\sin(\pi x), \qquad
    g(t)=\frac{1}{1+\bigl((t-t_0)/w\bigr)^2}, \qquad t_0 = 4.3,\quad w = 1.5,
\end{align}
whose analytic continuation has poles at $t = t_0 \pm w\,i$. Thus the equation is again $u_t-u_{xx}=f$ with zero Dirichlet data and initial data $u(0,x)=g(0)\sin(\pi x)$, but now
\begin{align*}
    f(t,x)=\bigl(g'(t)+\pi^2 g(t)\bigr)\sin(\pi x).
\end{align*}
The forcing is expanded in the time test basis exactly as before.

Splitting $[0,10]$ into blocks moves these poles farther away in each block's reference coordinate, increasing the admissible Bernstein parameter. The blocked rate is controlled by the worst local block, so if $\varrho_{\min}$ is the smallest admissible Bernstein parameter over the blocks, the predicted positive slope of the log-error curve against $N_t$ is $\log\varrho_{\min}$. Table \ref{tab:heatpole-rates} gives the distance-only guaranteed rate used in Theorem \ref{thm:error}, the sharper rate obtained from the actual mapped pole locations, and fitted rates from a separate exact-inflow diagnostic. Figure \ref{fig:heatpole} shows the practical effect: the global solve needs roughly $N_t \approx 112$ time coefficients to reach its error floor, whereas $L=4$ blocks reach machine precision with only $N_t = 16$ per block.
\begin{table}[t]
\centering
\caption{Bernstein-rate $\log\varrho$ for the problem in \eqref{eq:heatpole}. Larger values indicate faster asymptotic decay.}
\label{tab:heatpole-rates}
\begin{tabular}{lccc}
\toprule
Discretisation & guaranteed rate & mapped-pole rate & fitted rate\\
\midrule
$L=1$ (global) & $\ge 0.30$ & $0.30$ & $0.31\pm 0.03$\\
$L=2$ (blocks) & $\ge 0.57$ & $0.68$ & $0.71\pm 0.03$\\
$L=4$ (blocks) & $\ge 1.02$ & $1.05$ & $1.03\pm 0.13$\\
\bottomrule
\end{tabular}
\end{table}
\begin{figure}\centering
    \includegraphics[width=0.31\textwidth]{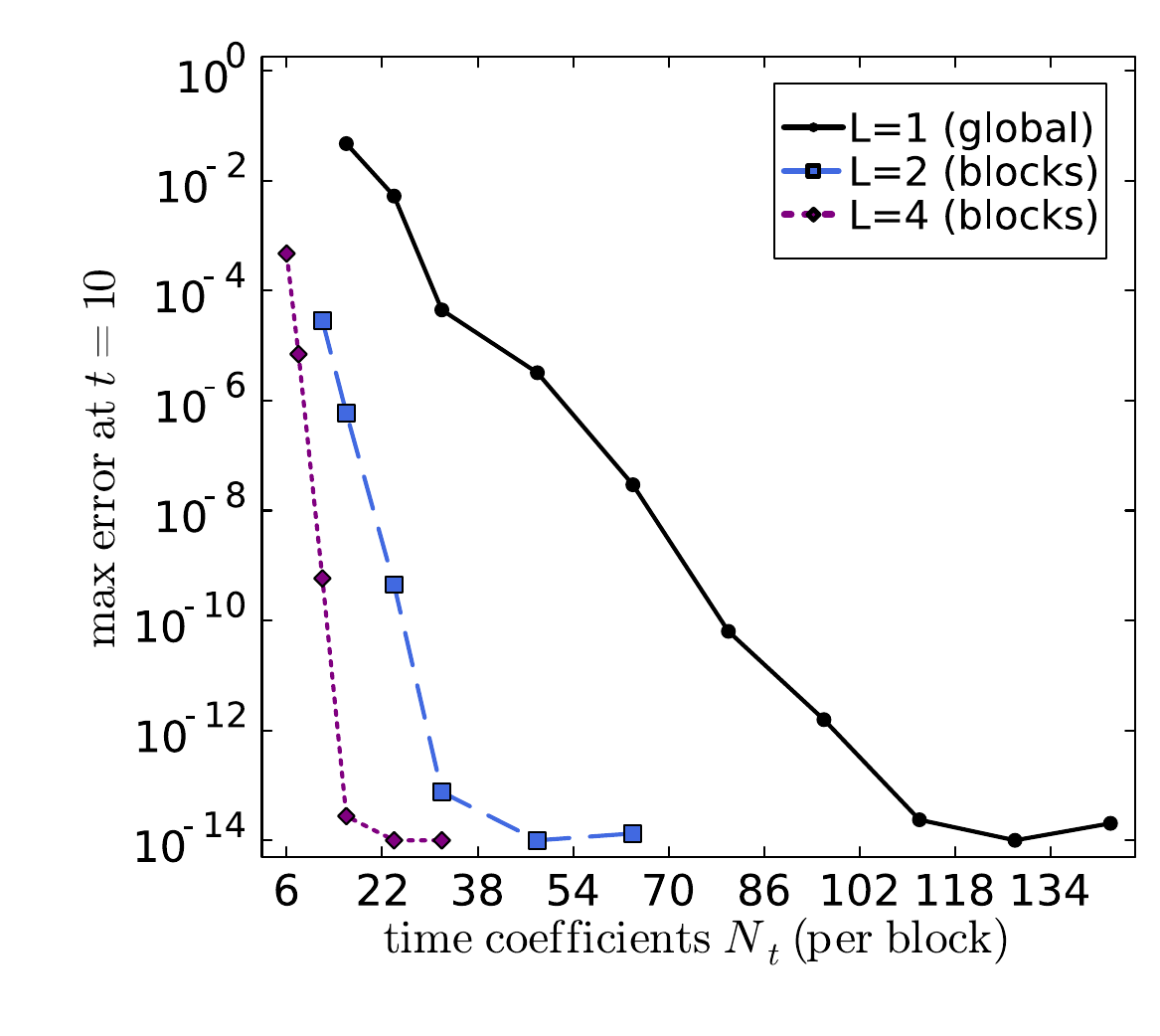}
    \includegraphics[width=0.31\textwidth]{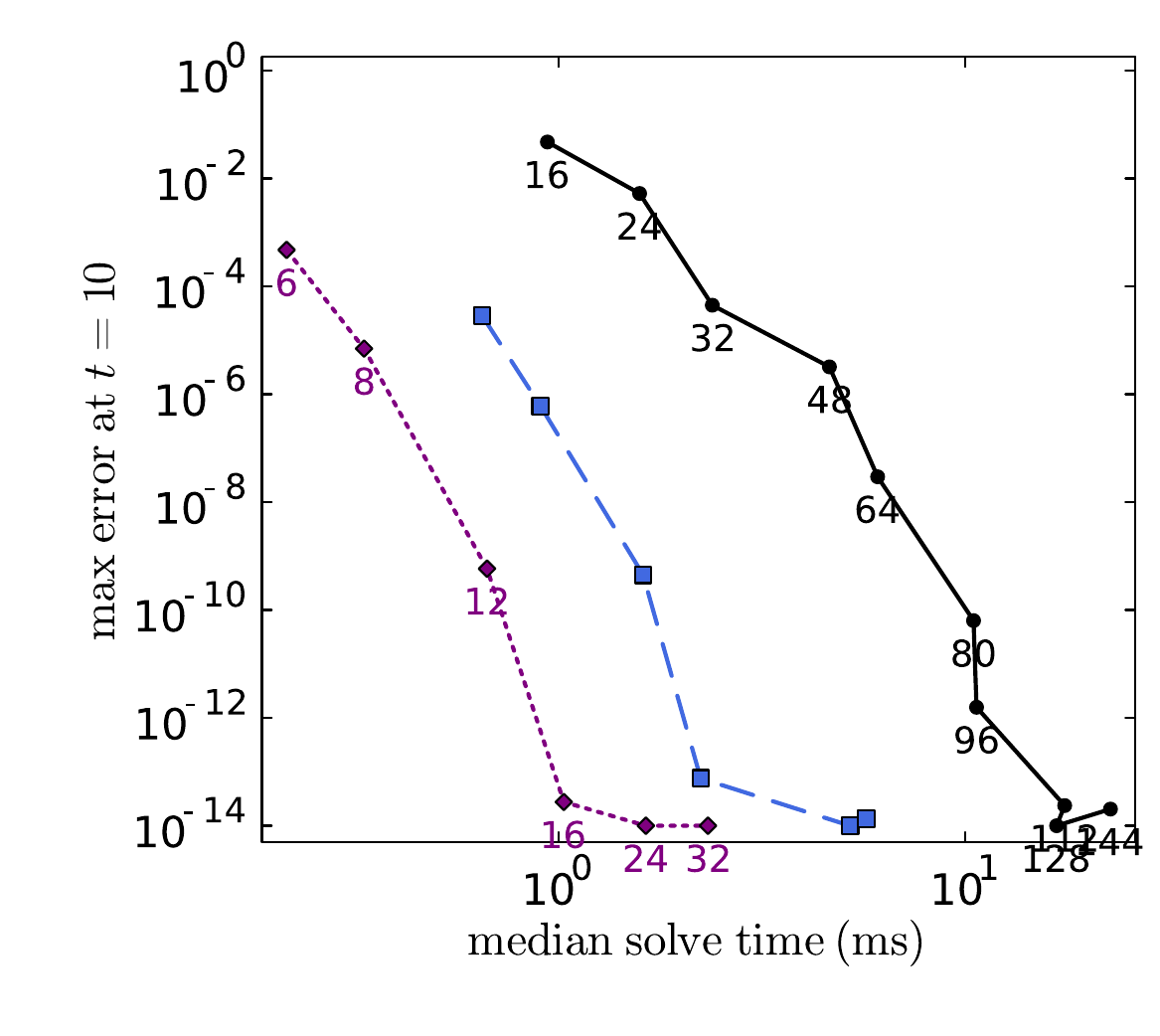}
    \includegraphics[width=0.31\textwidth]{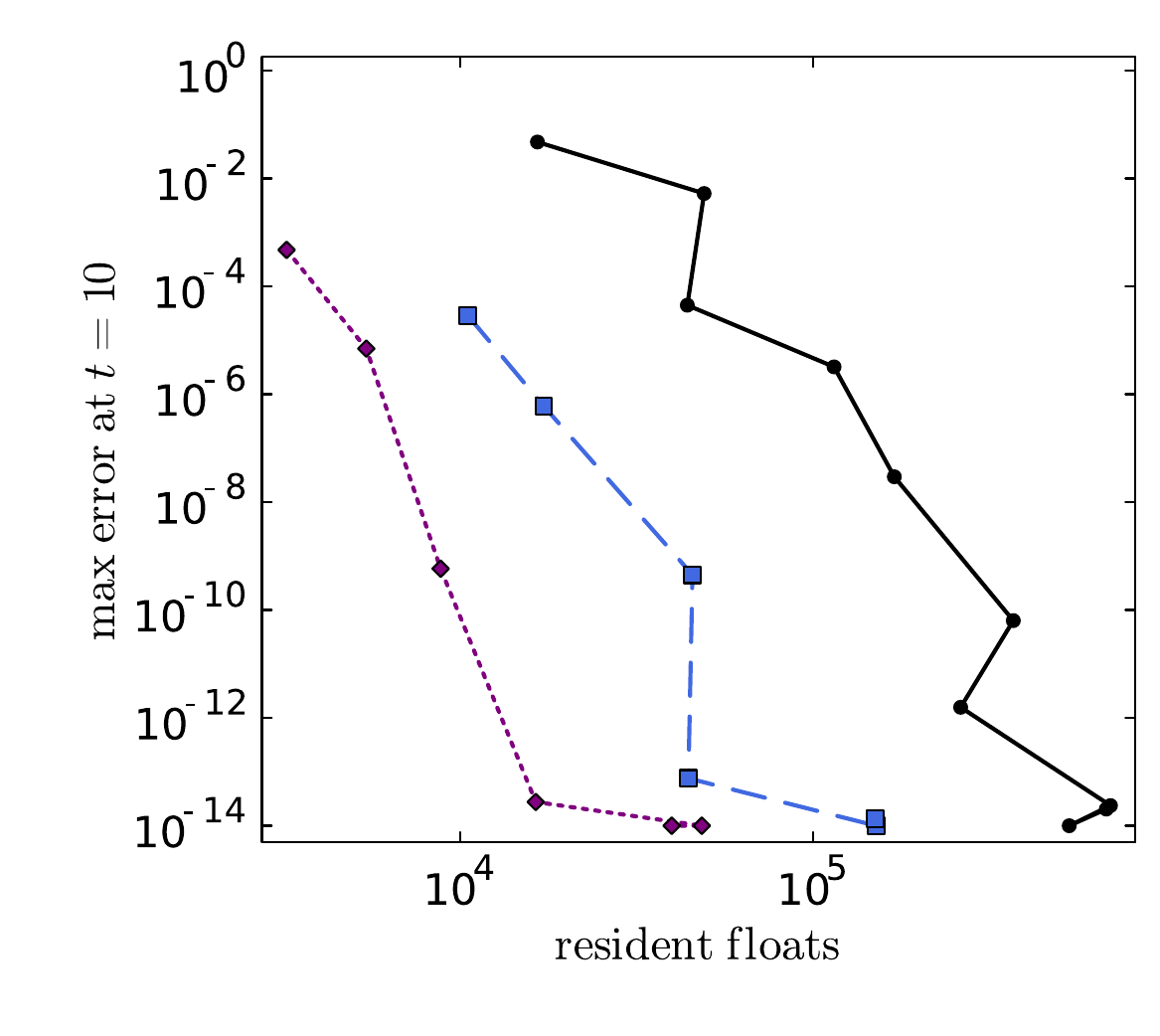}
    \caption{Heat equation with the complex-time singularity \eqref{eq:heatpole} over $[0,10]$, $N_x=24$. Columns: final-time error vs.\ $N_t$ (per block), solve time, and resident memory; bottom-left is better.}
    \label{fig:heatpole}
    \end{figure}
\subsection{Wave Equation}
We consider the wave equation
\begin{align}\label{eq:wave1p1}
    u_{tt} - u_{xx} = 0
\end{align}
on $\Omega = (0,1)$ with zero Dirichlet conditions. We compare two formulations. The derivative-peel formulation solves the scalar second-order equation and, in the blocked case, passes $u$ together with $u_t$ obtained by endpoint differentiation of the polynomial expansion. The first-order system formulation solves $u_t=v$, $v_t=u_{xx}$ and passes the value slices of both state components. This systemisation also changes the global discretisation, and in these experiments it improves the conditioning of both the global solve and the block-by-block solve.  Increasing the spatial resolution reduces the derivative-peel amplification in this example, but the comparison below uses the same spatial resolution in both formulations. To make the coefficient count less forgiving, we use moderately oscillatory data
\begin{align*}
    u(0,x) = \sin(2\pi x), \qquad u_t(0,x) = 0,
\end{align*}
whose exact solution is
\begin{align*}
    u(t,x) = \cos(2\pi t)\sin(2\pi x).
\end{align*}
Figure \ref{fig:wave1p1} compares final-time error, solve time, and resident memory over $[0,10]$, contrasting derivative peeling with systemisation, with the latter clearly preferred.
\begin{figure}\centering
    \includegraphics[width=\textwidth]{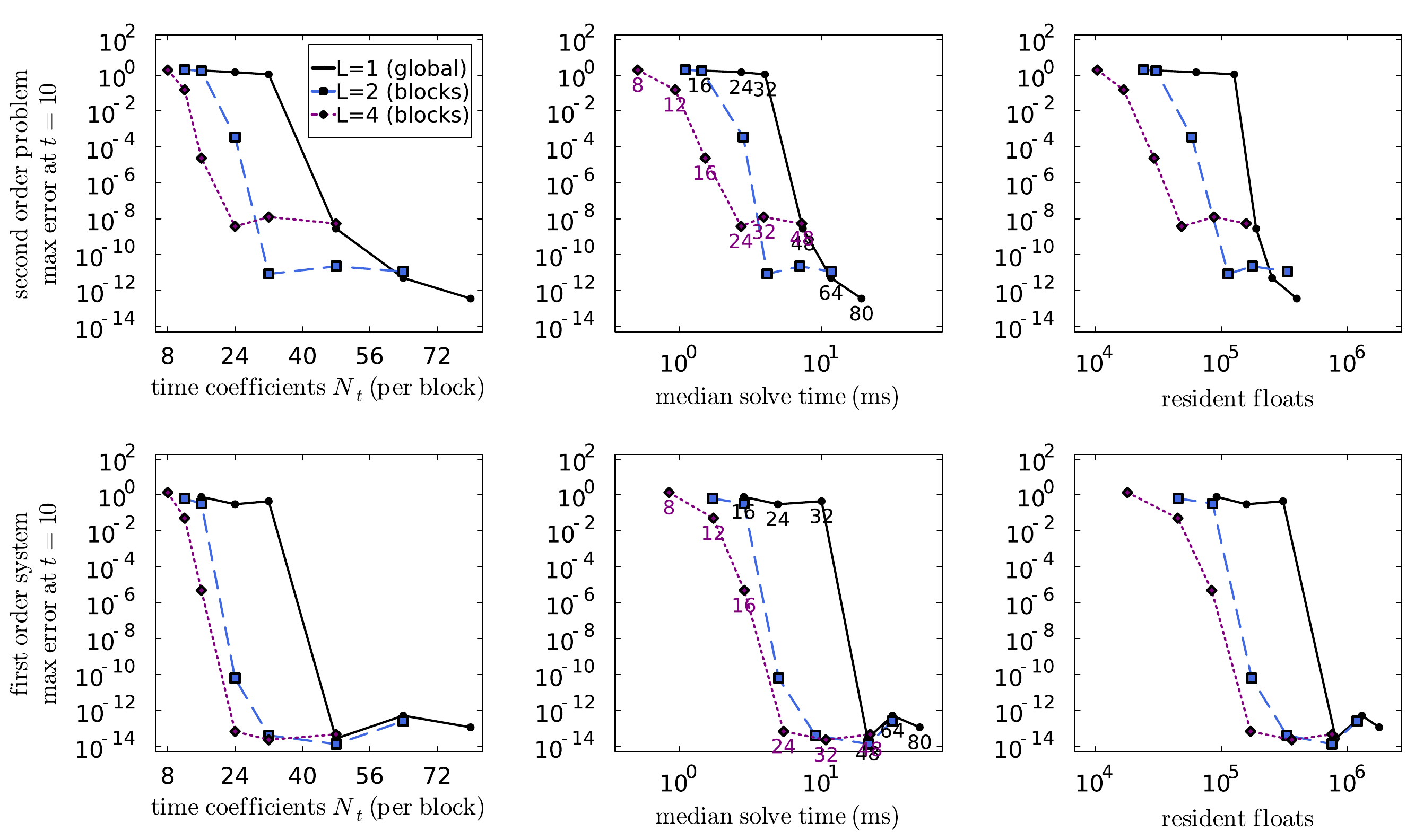}
    \caption{Wave equation \eqref{eq:wave1p1} over $[0,10]$, $u(0,x)=\sin(2\pi x)$, $N_x=32$. Rows: scalar derivative peeling vs.\ first-order system. Columns: final-time error vs.\ $N_t$ (per block), solve time, resident memory.}
    \label{fig:wave1p1}
    \end{figure}
\subsection{Klein--Gordon Equation}
We next add a zeroth-order term to obtain the Klein--Gordon equation
\begin{align}\label{eq:kg}
    u_{tt} - u_{xx} + m^2 u = 0,
\end{align}
on $\Omega = (0,1)$ with zero Dirichlet conditions, $m^2=20$, and initial data $u(0,x)=\sin(2\pi x)$, $u_t(0,x)=0$. The exact solution is
\begin{align*}
    u(t,x)=\cos\!\left(\sqrt{(2\pi)^2+20}\,t\right)\sin(2\pi x).
\end{align*}
Figure \ref{fig:kg} compares error, time, and memory over $[0,10]$. As for the wave equation, we compare scalar second-order blocks with derivative peeling against the first-order state $(u,v)=(u,u_t)$, including the corresponding global solve in each row. As before, the first-order system is the numerically preferable formulation.
\begin{figure}\centering
    \includegraphics[width=\textwidth]{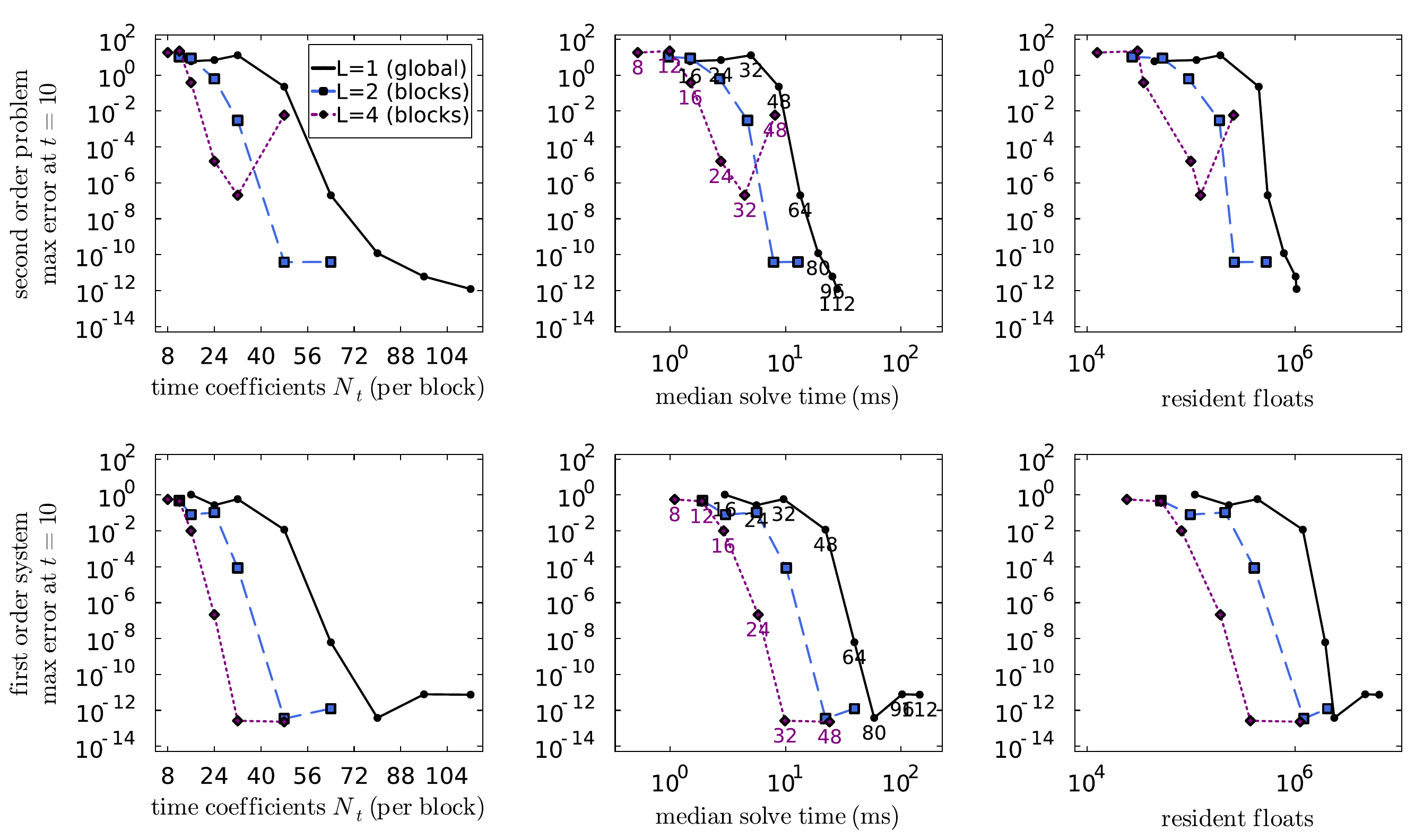}
    \caption{Klein--Gordon \eqref{eq:kg} over $[0,10]$, $m^2=20$, $u(0,x)=\sin(2\pi x)$, $N_x=32$. Rows: scalar derivative peeling vs.\ first-order system. Columns: final-time error vs.\ $N_t$ (per block), solve time, resident memory.}
    \label{fig:kg}
    \end{figure}

The bounded-horizon comparison is conservative because the memory advantage is most pronounced over long horizons. We therefore repeat the first-order Klein--Gordon solve to $T=1000$, about $1227$ periods, using $L=50$ and $L=100$ uniform blocks with $N_x=20$ spatial coefficients, the smallest truncation that resolves the spatial profile. Over-resolving the space enlarges the per-block operator and, over many blocks, slightly amplifies the accumulated roundoff. Figures \ref{fig:kg-deeptime} and \ref{fig:kg-deeptime-stability} show that the blocked method with $L=100$ reaches $\approx 2\times10^{-12}$ final-time error and keeps the passed-slice error below $10^{-11}$ over the full horizon, consistent with the mild accumulation predicted by Theorem \ref{thm:error}. In the convergence sweep, errors below $10^{-10}$ are already reached with $N_t=128$ for $L=50$ and $N_t=72$ for $L=100$, using resident memory around $3.6\times10^6$ and $9.5\times10^5$ floating-point scalars. Extrapolating the required resolution per period from the converged $L=100$ blocked run expects about $100\times72=7200$ time coefficients for a comparable single global interval, which cannot be ran on the test machine. A much smaller global solve attempt at $N_t=256$ remained at $\mathcal{O}(1)$ error. We therefore do not plot the global method in this regime.
\begin{figure}\centering
    \includegraphics[width=0.31\textwidth]{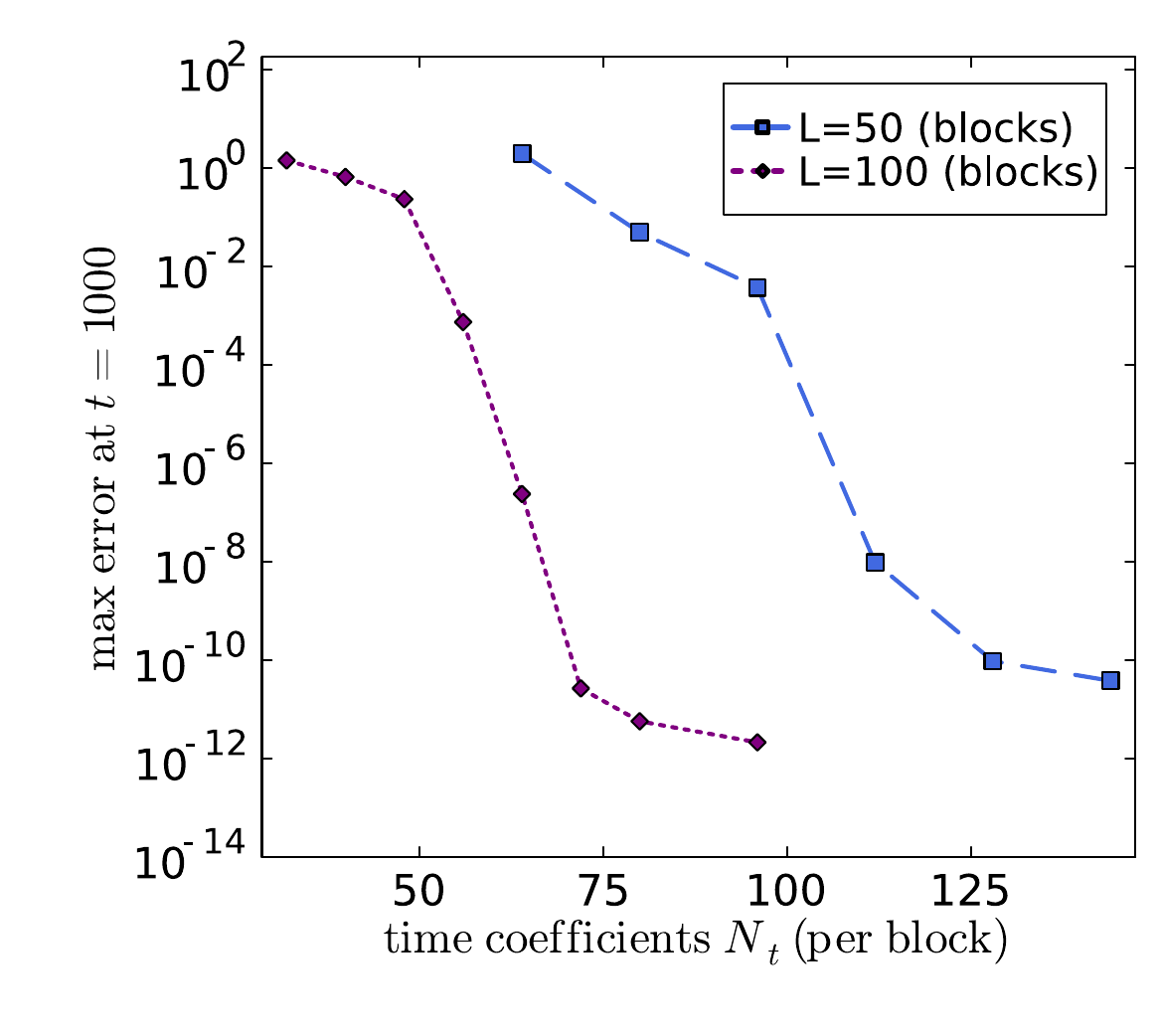}
    \includegraphics[width=0.31\textwidth]{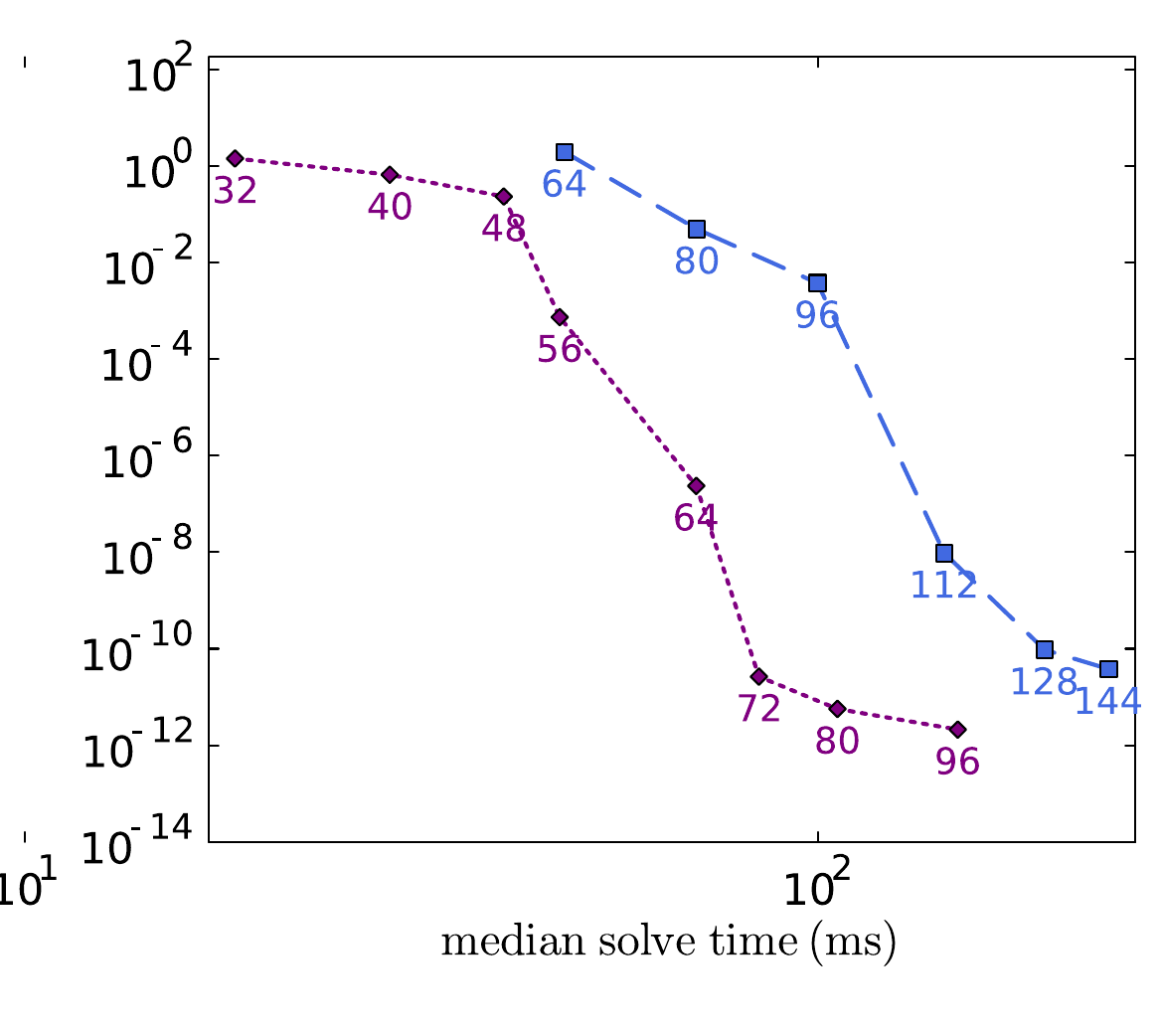}
    \includegraphics[width=0.31\textwidth]{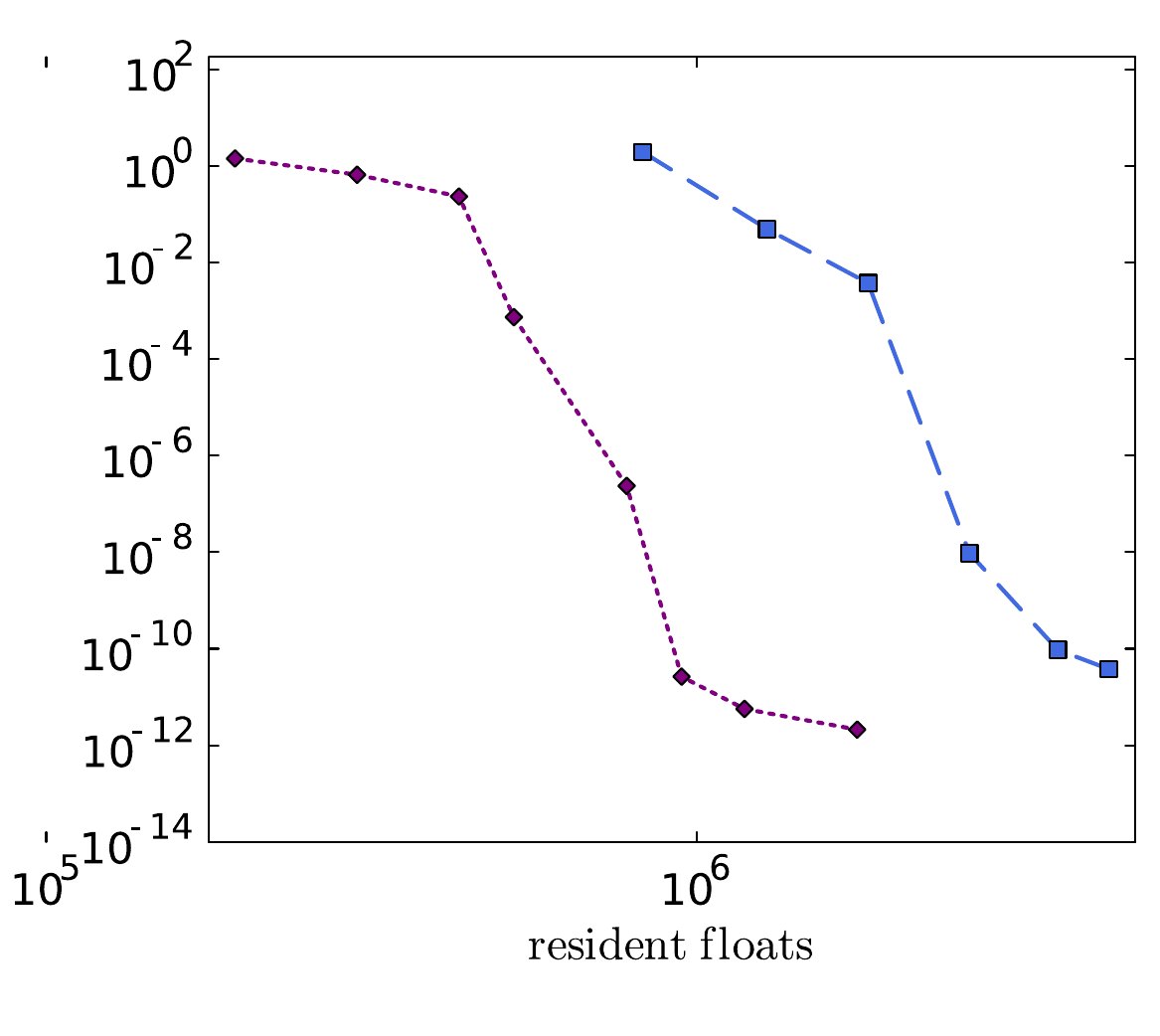}
    \caption{Deep-time Klein--Gordon \eqref{eq:kg} to $T=1000$ ($\approx1227$ periods), $m^2=20$, $N_x=20$, first-order system, $L=50$ and $L=100$ blocks. Columns: final-time error vs.\ $N_t$, solve time, and resident memory (independent of $T$); bottom-left is better.}
    \label{fig:kg-deeptime}
    \end{figure}
\begin{figure}\centering
    \includegraphics[width=0.31\textwidth]{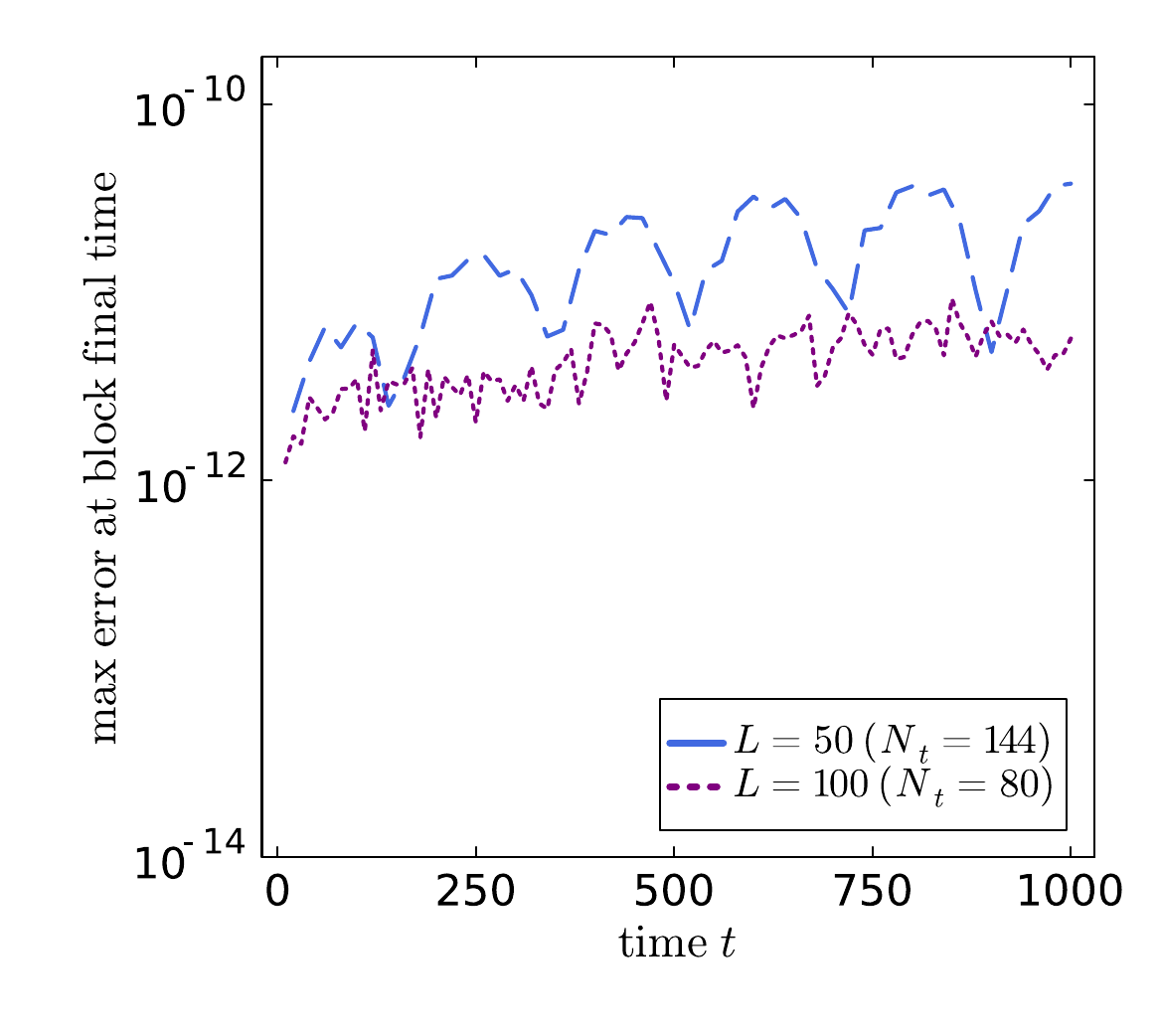}
    \caption{Passed-slice error at each block endpoint for the deep-time Klein--Gordon run \eqref{eq:kg} over $[0,1000]$, below $4\times10^{-11}$ throughout and below $10^{-11}$ for $L=100$. Empirical per-block amplification (Definition \ref{def:block-stability}) is $K\approx1.058$ ($L=50$) and $1.021$ ($L=100$).}
    \label{fig:kg-deeptime-stability}
    \end{figure}
\subsection{Disk Fractional Heat Equation}
To test a genuinely two-dimensional spatial geometry, we solve a manufactured forced fractional heat equation on the unit disk,
\begin{align}\label{eq:diskfractionalheat}
    u_t + (-\Delta)^\beta u = f, \qquad (x,y)\in\mathbb{D},
\end{align}
using Legendre polynomials in time and weighted generalised Zernike polynomials $(1-r^2)^\beta Z^{(\beta)}_{p,q}(x,y)$ in space. For $\beta=1$ this is the ordinary heat equation with homogeneous Dirichlet data encoded by the factor $(1-r^2)$. For $0<\beta<1$ we use the full-space fractional Laplacian applied to the zero extension outside the disk,
\[
    (-\Delta)^\beta v(x)
    =
    c_{d,\beta}\,\operatorname{PV}\!\int_{\mathbb{R}^d}
    \frac{v(x)-v(y)}{|x-y|^{d+2\beta}}\,dy,
    \qquad
    c_{d,\beta}
    =
    \frac{4^\beta\Gamma(d/2+\beta)}{\pi^{d/2}|\Gamma(-\beta)|},
\]
one of the equivalent standard definitions of the fractional Laplace operator \cite{kwasnicki_2017_ten_definitions}. Thus $\beta=1/2$ gives a genuinely fractional disk operator. The action of $(-\Delta)^\beta$ is diagonal in the weighted Zernike representation, following explicit formulae for ball and Zernike bases \cite{gutleb_carrillo_olver_2023_power_law,gutleb_papadopoulos_2023_fractional_laplacians}. The remaining weighted-to-unweighted conversion multiplies by $(1-r^2)^\beta$. For $\beta=1$ this is a polynomial multiplication, so the conversion is banded. For $\beta=1/2$ the factor $(1-r^2)^{1/2}$ is not a polynomial, so the conversion is not banded. The radial weight couples only modes of equal angular momentum, leaving a conversion that is block-diagonal in angular momentum with a dense radial block in each. We form it numerically and drop assembled entries of magnitude at most $10^{-15}$, which removes only roundoff-scale fill. While this conversion can in principle be constructed sparsely, we do not pursue that implementation detail here.

We take a non-radial manufactured solution $u(t,x,y)=e^{-(4+2\beta)t}\Phi_\beta(x,y)$, where $\Phi_\beta$ is a fixed sparse combination of weighted-Zernike modes, and use the complete triangular Zernike truncation with $N_Z=105$ coefficients. In these runs $T=10$, the truncation has $14$ triangular levels, and $\Phi_\beta$ has nonzero coefficients $(1,-.20,.12,-.08,.05,.035,-.030,.026,-.022,.018,-.015,.012,-.010,.008,-.006)$ at coefficient indices $(1,2,4,7,11,16,22,29,37,46,56,67,79,88,100)$. The forcing coefficients are obtained by expanding $f=u_t+(-\Delta)^\beta u$ directly in the time test basis and the Zernike representation. Figure \ref{fig:diskfractionalheat-zernike-summary} compares $\beta=1$ and $\beta=1/2$ for a global solve and for $L=2$ and $L=4$ blocks.

The operator displays in Figure \ref{fig:diskfractionalheat-zernike-spy} come from the sparse matrices used by the sparse LU solve, whose PDE rows have the tensor-product form
\begin{align*}
    C_{\mathbb{D}}\otimes D_t + L_{\mathbb{D}}\otimes S_t,
\end{align*}
where $D_t$ is the sparse time derivative, $S_t$ is the sparse conversion into the common test basis, $C_{\mathbb{D}}$ is the finite Zernike conversion matrix, and $L_{\mathbb{D}}$ is the finite fractional-Laplacian matrix. For this disk experiment only, timing is measured from the assembled matrix onward, so matrix assembly is not counted. This is mildly conservative for the blocked comparison, since generating the larger global operator would cost more than generating a smaller block operator even in a hypothetically optimal directly sparse construction.
\begin{figure}\centering
    \includegraphics[width=\textwidth]{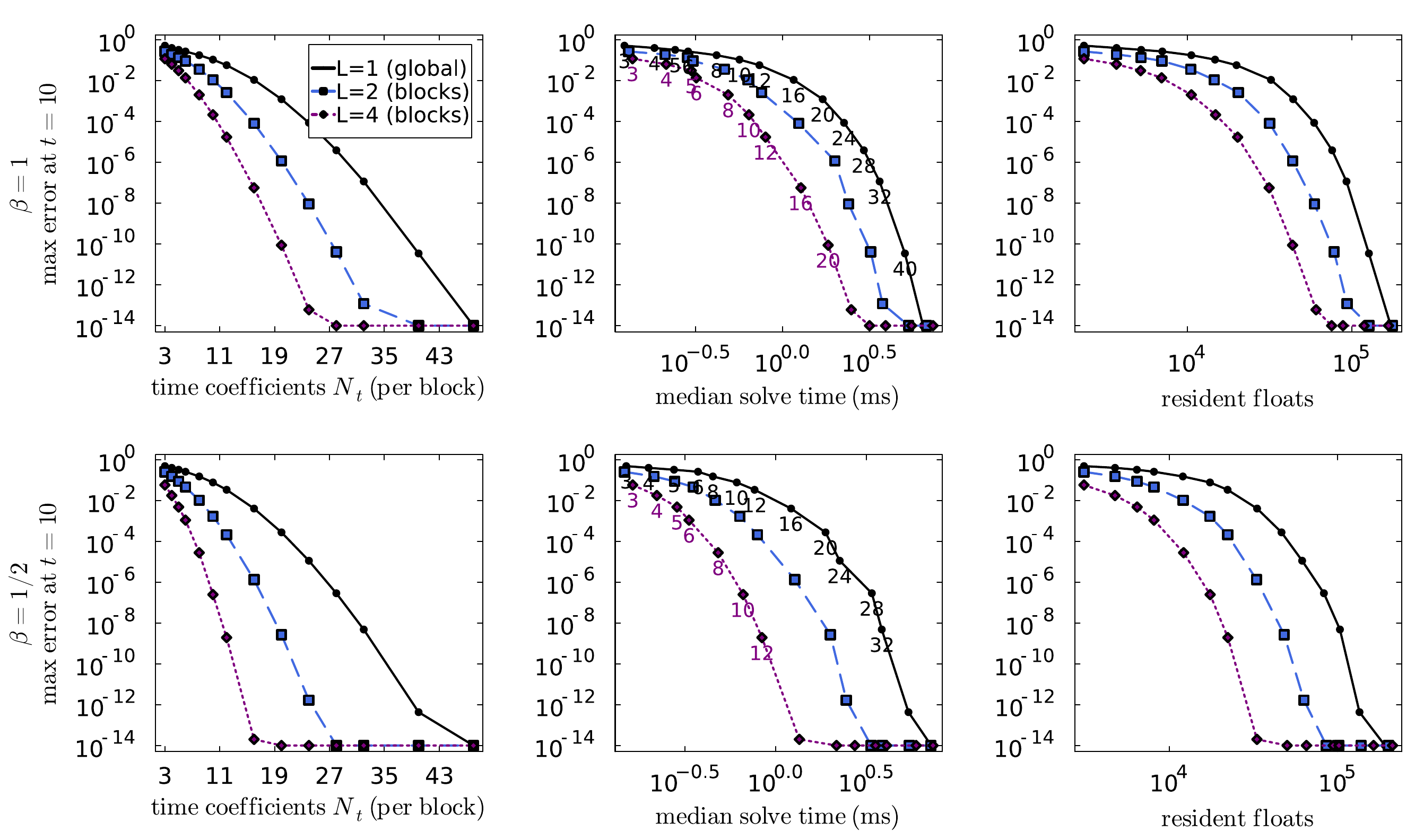}
    \caption{Disk fractional heat \eqref{eq:diskfractionalheat} over $[0,10]$, weighted Zernike basis with $N_Z=105$. Rows: $\beta=1$ and $\beta=1/2$. Columns: final-time error vs.\ $N_t$, solve time, and resident memory; bottom-left is better.}
    \label{fig:diskfractionalheat-zernike-summary}
    \end{figure}
 \begin{figure}\centering
     \subfloat[$\beta=1$]
    {{ \centering \includegraphics[width=0.31\textwidth]{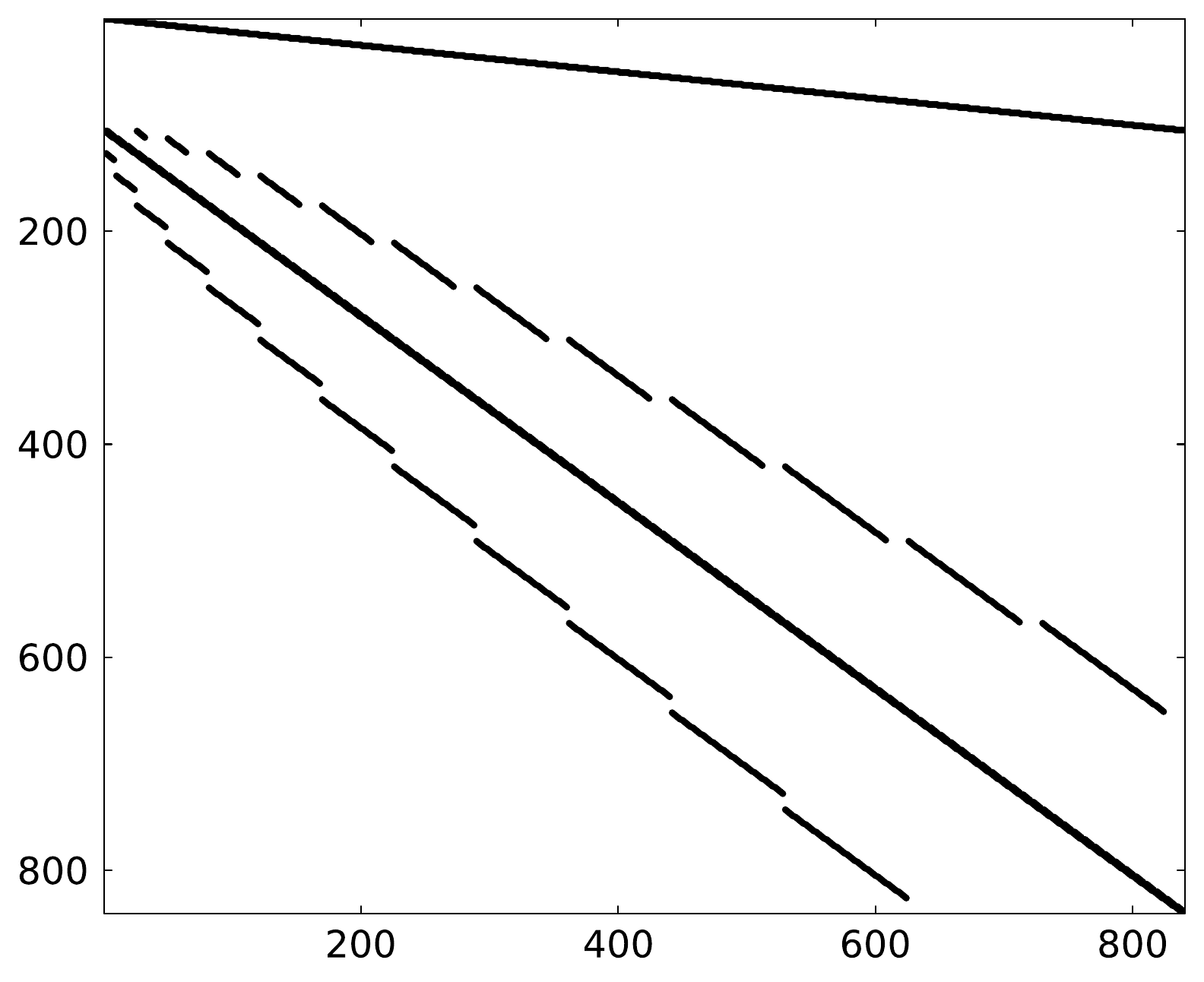} }}
     \subfloat[$\beta=1/2$]
    {{ \centering \includegraphics[width=0.31\textwidth]{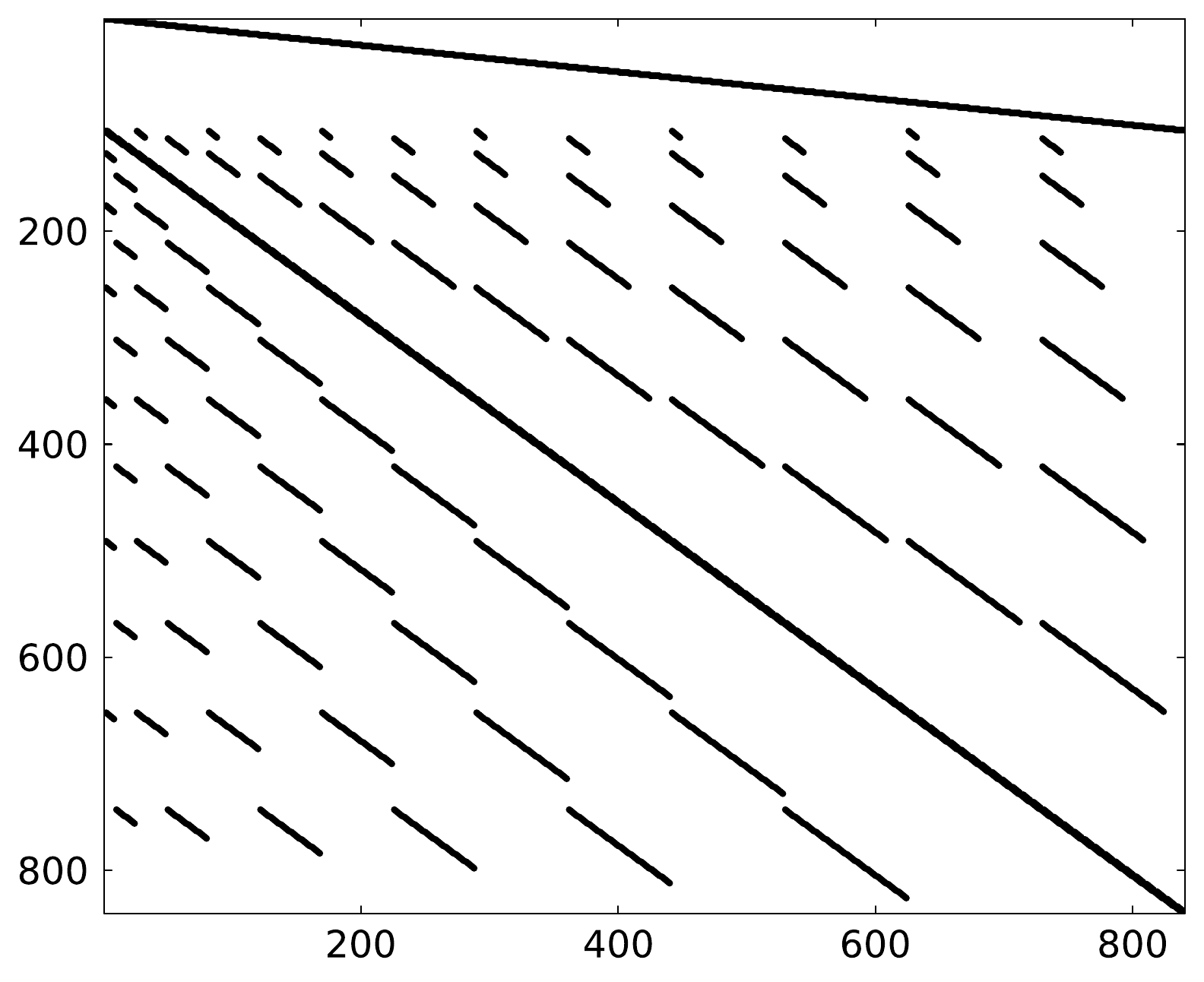} }}
    \caption{Nonzero patterns of Legendre--Zernike single-block operators ($N_t=8$, $N_Z=105$). Only the $\beta=1/2$ operator is thresholded, dropping entries below $10^{-15}$.}
    \label{fig:diskfractionalheat-zernike-spy}
    \end{figure}
\section{Discussion}
We have highlighted that sparse space-time spectral blocks admit exact coefficient-level peel and pass. This enables a form of slab marching, or time stepping, for sparse space-time spectral methods. One solves a local space-time block, efficiently peels off the final slice, and passes those coefficients as data for the next block. For higher-order equations this can be done either by passing endpoint derivative slices or, often more robustly, by rewriting the equation as a first-order system and passing the value slices of the state components. The method can therefore forget the past while retaining spectral accuracy, reducing resident memory to that of a single block and often lowering the time coefficient count required in each solve for temporally resolved long-time problems.\\
There is naturally room for implementation-level optimisation: the structured tensor-product matrices shown in Figures \ref{fig:operator-sparsity} and \ref{fig:diskfractionalheat-zernike-spy} need not be in the best ordering for direct solvers, and ordering choices such as interweaving are already part of sparse spectral linear algebra. Further work should also explore whether direct treatment of higher order systems can be saved in this framework, though systemisation is a straightforward and competitive alternative as shown in the numerical experiments.

\backmatter

\bmhead{Acknowledgments}
The author thanks Avleen Kaur and Tianyi Pu for helpful discussions on space-time spectral methods in 2023 and Mika{\"e}l Slevinsky for the invitation to CAIMS 2026 which prompted the completion of this work.

\bmhead{Data availability}
No external data was used in this work.

\bmhead{Code availability}
The Julia code used to generate the numerical experiments and figures is provided in the companion repository \url{https://github.com/TSGut/SpacetimeSpectralPassExamples.jl} and archived at \cite{spacetimespectralpass_software}.

\bibliography{references}

\end{document}